\documentclass{amsart}

\usepackage{a4wide}
\usepackage{amsmath}
\usepackage{enumerate}

\usepackage{cancel}

\usepackage{amsmath,amsthm}
\usepackage{amsfonts}
\usepackage{amssymb}
\usepackage{graphicx, color, soul}
\usepackage[T1]{fontenc}
\usepackage{verbatim}

\newtheorem*{theorem*}{Theorem}
\newtheorem*{cor*}{Corollary}

\theoremstyle{plain}

\newtheorem{theorem}{Theorem}[section]
\newtheorem{cor}[theorem]{Corollary}
\newtheorem{prop}[theorem]{Proposition}

\newtheorem{lem}[theorem]{Lemma}

{\rm}{\rm}

\theoremstyle{definition}

\newtheorem{rem}[theorem]{Remark}

\newcommand{\Hol}{\operatorname{Hol}}
\newcommand{\GL}{\operatorname{GL}}
\newcommand{\SO}{\operatorname{SO}}

\newcommand{\Or}{\operatorname{O}}
\newcommand{\SU}{\operatorname{SU}}
\newcommand{\Un}{\operatorname{U}}

\newcommand{\Spin}{\operatorname{Spin}}

\newcommand{\Ric}{\operatorname{Ric}}

\def\R{\mathcal{R}}
\def\P{\mathcal{P}}

\newcommand{\pr}{\operatorname{pr}}
\newcommand{\id}{\operatorname{id}}

\def\so{\mathfrak{so}}
\def\su{\mathfrak{su}}
\def\sp{\mathfrak{sp}}

\def\u{\mathfrak{u}}
\def\t{\mathfrak{t}}

\newcommand{\h}{\mathfrak{h}}
\newcommand{\hol}{\mathfrak{hol}}

\newcommand{\bt}{{\boldsymbol{\tau}}}

\newcommand{\Real}{\mathbb{R}}
\newcommand{\Co}{\mathbb{C}}

\newcommand{\Su}{\mathfrak{S}}

\usepackage{euscript,color}

\newcommand{\1}{\mathbf{1}}
\newcommand{\tr}{\mathrm{tr}}
\newcommand{\ad}{\mathrm{ad}}
\newcommand{\Ad}{\mathrm{Ad}}

\newcommand{\p}{\mathfrak{p}}
\renewcommand{\k}{\mathfrak{k}}
\newcommand{\g}{\mathfrak{g}}
\newcommand{\Iso}{\mathrm{Iso}}
\newcommand{\D}{{\mathcal D}}
\newcommand{\E}{{E}}

\renewcommand{\d}{\mathrm{d}}
\renewcommand{\L}{{\mathcal L}}
\newcommand{\+}{\oplus}
\newcommand{\N}{N}
\newcommand{\bnab}{\overline{\nabla}}
\newcommand{\tnab}{\widetilde{\nabla}}

 \newcommand{\nbt}{\nabla^\bt}
 \newcommand{\Rbt}{R^\bt}
 \newcommand{\nbbt}{\overline{\nabla}^\bt}
  \newcommand{\nbW}{\overline{\nabla}}
  \newcommand{\Rbbt}{\overline{R}^\bt}

\begin{document}
	
	\title{Holonomy in pseudo-Hermitian geometry}
\setlength{\footskip}{.8cm}	
	
	\author[Anton S. Galaev]{Anton S. Galaev}
	\thanks{Corresponding author Email: anton.galaev@uhk.cz}
	\address[Galaev]{University of Hradec Kr\'alov\'e, Faculty of Science, Department of Mathematics, Rokitansk\'eho 62, 500~03 Hradec Kr\'alov\'e, Czech Republic}
	\email{anton.galaev@uhk.cz}

	\author[Thomas Leistner]{Thomas Leistner}\address[Leistner]{School of Computer and Mathematical Sciences, University of Adelaide, SA~5005, Australia}\email{thomas.leistner@adelaide.edu.au}

	\author[Felipe Leitner]{Felipe Leitner}\address[Leitner]{Institut f\"ur Mathematik und Informatik, Walther-Rathenau-Str. 47, 17489 Greifswald, Germany}\email{felipe.leitner@uni-greifswald.de}

	\begin{abstract}
We study the holonomy that is associated to a sub-Riemannian  structure defined 
 on the kernel of a global contact form. This includes the holonomy of  Schouten's horizontal connection as well as of the adapted connection, both canonical invariants of the structure.  	Under a condition on the torsion of the structure, we show that they are either equal or that the former is a codimension one normal subgroup of the latter. Furthermore, we establish a close relation to Riemannian holonomy, which yields a complete holonomy classification in the torsion-free case. For the main result we focus on the special case of pseudo-Hermitian structures and give a classification of holonomy algebras for both the Schouten and the adapted connection. Based on this, we derive a classification of symmetric sub-Riemannian structures and of of those holonomy groups that admit parallel spinors. Finally we exhibit a relation between locally symmetric sub-Riemannian contact structures and locally homogeneous Riemannian structures.

\medskip
		
		{\bf Keywords}: holonomy; pseudo-Hermitian structures;  sub-Riemannian structures; Schouten connection; sub-symmetric spaces; parallel spinors

\medskip
		
		{\bf AMS Mathematics Subject Classification 2020:} 53C29; 53C17; 15A66 
		
		
	\end{abstract}

	\maketitle
	
\setcounter{tocdepth}{1}	
\tableofcontents

\section{Introduction}
The classification of holonomy groups of Riemannian manifolds by M.~Berger in \cite{berger55} is a classical result that has a many consequences and applications, see for example  \cite{Besse,Bryant92,bryant99,Joyce07}.
For CR~manifolds the situation is very different, as it is stated  in the well-known monograph  on the topic \cite[p. 330]{DT2006}:
\begin{quote}
	{\it ,,A systematic study of the (pseudo-Hermitian) holonomy of a CR manifold is still missing in the present-day mathematical literature.''}\end{quote}
The goal of this paper is to fill this gap by providing a classification of holonomy algebras of pseudo-Hermitian spaces. Let us describe its contents and its  results. 

In Section \ref{secHorCon}, we review basic facts on  contact sub-Riemannian manifolds $(M,\theta,g)$, which we defines as given by  a fixed contact form  $\theta$ on a smooth manifold $M$ and a is a sub-Riemannian metric  $g$ on the contact distribution $\D=\ker\theta$. An invariant of such a structure is the Lie derivative of $g$ along the Reeb vector field $\xi$, denoted by $\bt=\frac{1}{2}\mathcal L_\xi g$ and called the {\em sub-torsion} or {\em torsion} of $(M,\theta,g)$.
A vector field or a curve is called {\em horizontal} if it is tangent to $\D$.
The analogue for the Levi-Civita connection is the {\em Schouten connection} $\nabla$, which is a partial connection along horizontal vector fields.  Then we consider extended connections $\nabla^N$, i.e., connections on the vector bundle $\D$ over $M$ that arise from the Schouten connection and an endomorphism $N$ of $\D$. For our approach two extended connections will be crucial: the well-known {\em adapted connection} $\nabla^\bt$ and the lesser known {\em Wagner connection} $\nabla^W$,  which originates from \cite{Wagner41}. 
Both connections preserve the metric $g$ on $\D$, and we will write down the expressions for the curvature tensors of these connections. 
The importance of the Wagner connection for holonomy was realized by the first author  in \cite{GalHolK}, and it continues to provide a key tool in the current paper. 

In Section \ref{secHol}, we consider the related  holonomy groups. The holonomy group $\Hol(\nabla)$ of the Schouten connection $\nabla$ is defined as the group of parallel transports along horizontal loops,~\cite{FGR}. The holonomy groups $\Hol(\nabla^\bt)$ and $\Hol(\nabla^W)$ are defined, as usual, as the groups of parallel transports along all loops.
We recall the Ambrose--Singer type holonomy theorems from \cite{CGJK,FGR} for the group $\Hol(\nabla)$, and we stress that the results in \cite{CGJK,FGR} imply the equality   \[\Hol(\nabla^W)=\Hol(\nabla).\] 
Then we generalize the results from \cite{GalHolK}. There the first author considered simply connected {\em K-contact manifolds}, i.e., those with vanishing torsion, $\bt=0$, and proved 
that if the Reeb vector field is complete, then 
\begin{equation}
\label{dicho}
\text{either  $\Hol(\nabla)=\Hol(\nabla^\bt)$, or $\Hol(\nabla)\subset\Hol(\nabla^\bt)$ is a codimension-one normal subgroup,}
\end{equation} 
and that $\Hol(\nabla^\bt)$ is a  holonomy group of a Riemannian manifold.
As first generalisation of these results, here we show the following:
\begin{enumerate}[(a)]
	\item \label{a} The dichotomy~(\ref{dicho}) holds when the torsion is a Codazzi-tensor (a differential condition on $\bt$,  completeness of $\xi$ is not assumed anymore, 
	see Theorem~\ref{thHolnabl}).  
	\item \label{b} When $\bt$ is not Codazzi and
	$\hol(\nabla^\bt)$ is irreducible, it is the holonomy algebra of a Riemannian manifold or equal to $\sp(k)\+ \u(1)$ (see Proposition~\ref{propRieamHolor}, and note that the latter is not Berger algebra when $k\geq 2$). 
\end{enumerate}
In order to prove the first result, for any extended connection $\nabla^N$, we define a modified holonomy group $\widetilde\Hol(\nabla^N)$, which consists of parallel transports along loops $\gamma$  for which the integral of $\theta$ along $\gamma$ is zero. We prove that if the Reeb vector field $\xi$ is complete, then there is a dichotomy: either  $\widetilde\Hol(\nabla^N)=\Hol(\nabla^N)$, or $\widetilde\Hol(\nabla^N)\subset\Hol(\nabla^N)$ is a codimension-one Lie subgroup.  
This is then used to prove (\ref{a}) and (\ref{b}) and  we are able to drop the completeness assumption on $\xi$. 
The strengthened results (\ref{a}) and (\ref{b}) go beyond the K-contact case, but in particular {\em complete} the classification of the holonomy algebras $\hol(\nabla)$ of K-contact sub-Riemannian manifolds:
they are exhausted by the holonomy algebras of Riemannian manifolds and their codimension-one ideals. Thus, we are left with the holonomy classification problem for contact sub-Riemannian manifolds with $\bt\neq 0$. We solve this problem for the case of pseudo-Hermitian manifolds.

In Section \ref{secPsHerm}, we recall the definition and the main properties of {\em pseudo-Hermitian manifolds}. In particular, we show that in that case the adapted connection $\nabla^\bt$ coincides with the Tanaka--Webster connection, and $\bt$ coincides with the Tanaka--Webster torsion. In Section \ref{secpsHermHolClass}, we recall the definition of {\em contact sub-Riemannian symmetric spaces},  defined in \cite{FG}, and related notions.  They are examples of pseudo-Hermitian manifolds. In \cite{BFG,FG,FGR}, simply connected contact sub-Riemannian symmetric spaces were classified and  it was shown that they satisfy the dichotomy~(\ref{dicho}). In Section \ref{secLocsubsym}, we study locally symmetric
contact sub-Riemannian spaces. We show that these spaces may be characterized as contact sub-Riemannian manifolds for which the  Wagner connection $\nabla^W$ is a Riemannian  Ambrose--Singer connection, as in \cite{AmbroseSinger58} or \cite{tricerri-vanhecke83},  i.e., a homogeneous structure, which implies that the Riemannian metric $\theta^2+g$ on $M$ is locally homogeneous. 

Section \ref{SecClass} contains the main result of the paper: Theorem~\ref{ThClass} gives the classification of holonomy algebras of pseudo-Hermitian spaces with non-vanishing torsion. 
For a pseudo-Hermitian manifold of dimension $2m+1\geq 7$ with $\bt\neq 0$
there are the following possibilities for the  holonomy algebras $\hol(\nabla^{\bt})$ of the Tanaka--Webster connection and $\hol(\nabla)$ of the Schouten connection:
\begin{itemize}
	\item[ 1.] $\hol(\nabla^{\bt})=\u(m)$, and either $\hol(\nabla)=\u(m)$ or $\hol(\nabla)=\su(m)$;
	\item[ 2.] $\hol(\nabla^{\bt})=\hol(\nabla)=\su(m)$;
	\item[ 3.] $\hol(\nabla^{\bt})=\so(m)\oplus\so(2)$ and $\hol(\nabla)=\so(m)$; 
	\item[ 4.] $\hol(\nabla^{\bt})=\hol(\nabla)=\so(m)$.
\end{itemize}
Here $\so(m)\subset\su(m)$ is the subalgebra preserving a Lagrangian subspace of $\Real^{2m}$. Note that the classification of holonomy algebras of torsion-free sub-Riemannian contact structures, i.e.~with $\bt=0$,   was given in \cite{GalHolK} and includes the torsion-free pseudo-Hermitian case. We review  these results  in Section~\ref{secbt0}.

In Section \ref{SecCons}, we prove that in the above list, the holonomy 
$\hol(\nabla)=\so(m)$ corresponds to locally sub-symmetric spaces. Moreover, for a pseudo-Hermitian manifold with $\bt\neq 0$,  the condition $\nabla\bt=0$ is the equivalent condition for the space to be locally sub-symmetric. Next, a pseudo-Hermitian manifold with $\nabla\bt\neq 0$ is pseudo-Einstein if and only if  $\hol(\nabla)=\su(m)$. We give a construction of pseudo-Hermitian spaces with $\hol(\nabla)=\su(m)$ using a deformation of pseudoconvex hypersurfaces in $\Co^{m+1}$.

In Section \ref{secparaldtheta}  we consider contact sub-Riemannian spaces $(M,\theta,g)$ with $\nabla\d\theta =0$, which are a generalization of pseudo-Hermitian  spaces. We prove that if $\bt\neq 0$, then the condition $\nabla\d\theta =0$ forces $(M,\theta,g)$ to be  pseudo-Hermitian, after multiplication of $g$ by a constant.

In Section \ref{subsymholclass-sec}, we deduce the classification of simply connected sub-symmetric spaces from the classification of the holonomy algebras. Our approach simplifies the proof from \cite{BFG,FG,FGR}, since we use the information about the holonomy algebra and the structure of the curvature tensor.

In Section \ref{secSpinors}, we describe holonomy groups of simply connected pseudo-hermitian manifolds that admit parallel spinors.

\vskip0.3cm

{\bf Acknowledgements.}
A.G.~was supported by the project GF24-10031K of Czech Science Foundation (GA\v{C}R). A.G.~and T.L.~would like to thank Vicente Cort\'es and the University of Hamburg for their hospitality.

\section{Preliminaries on  sub-Riemannian contact structures}\label{secHorCon}

\subsection{The Schouten connection and its curvature}Let $(M,\theta,g)$ a contact sub-Riemannian manifold, 
 i.e.,~$\theta$ is a 
 contact form on a smooth manifold $M$, and $g$ is a sub-Riemannian metric on the contact distribution $\D=\ker\theta$.  We assume that $\dim M=2m+1\geq 5$. 

 Denote by $\xi$ the corresponding Reeb vector field on $M$ defined by the conditions 
$$\theta(\xi)=1,\quad \iota_\xi \d\theta=0.$$  
These conditions imply $\L_\xi\d\theta =0$ and that $\L_\xi=[\xi, \,.\, ]$ maps sections of $\D$ to sections of $\D$.
In particular,  $\L_\xi g$ is a well-defined section of $\odot^2\D$. We define the symmetric bilinear form $\bt$ of $\D$ by
\begin{equation}\label{defA} \bt(X,Y)=\tfrac{1}{2}(\mathcal{L}_\xi g)(X,Y),\quad \text{for all } X,Y\in\D.\end{equation}
We denote by the same symbol $\bt$ the corresponding symmetric endomorphism field, which is also called the {\em sub-torsion}, see \cite{FG}.

Denote by $\left<\xi\right>$ the line bundle with fibres $\Real\cdot \xi_p$, 
the decomposition
$$TM=\D\oplus \left<\xi\right>,$$
and the projections  \begin{equation}\label{projections}
\pi =\mathrm{Id}-\theta\otimes \xi :TM\longrightarrow \D,\quad \pi' =\theta \otimes \xi :TM\longrightarrow \left<\xi\right>.\end{equation} The projection $\pi$ satisfies
\begin{equation}
\label{Lpi}
\left[ \mathcal L_\xi, \pi \right]=0,
\end{equation}
because for all $X\in \Gamma(TM)$, we have
\[
\left[ \xi,\pi(X)\right]- \pi\left( \left[\xi,X\right]\right)
=
-\left( \d\theta(\xi,X) - X(\theta(\xi))\right)\xi=
0,
\]
as $\xi$ is a Reeb vector field.

{\it A horizontal  connection} or {\em partial connection} on the distribution $\D$ is a differential operator 
\begin{equation}\label{nablaD}
\nabla:\Gamma(\D) \longrightarrow \Gamma(\D^*\otimes \D)\end{equation} 
satisfying the Leibniz rule. A partial connection extends to  tensor products $\mathcal \D^{(r,s)}=\otimes^r \D^*\otimes \otimes^s\D$ of $\D$ in a natural way,
\[\nabla: \Gamma ( \mathcal \D^{(r,s)})\longrightarrow\Gamma(\mathcal \D^{(r+1,s)}).\]
The {\em Schouten connection} is a partial connection that is the analogue of the Levi-Civita connection, i.e.,~it is the unique partial connection satisfying the conditions 
\begin{eqnarray}
\label{metriccondit}
\nabla g&=&0,
\\
\label{torsioncondit}\nabla_X Y-\nabla_Y X-\pi[X,Y]&=&0, \quad\text{for all $X,Y\in \Gamma(\D)$.}
\end{eqnarray} It is defined by the Koszul-type formula
\begin{equation}\label{SchoutenCon}
2g(\nabla_X Y,Z)= Xg(Y,Z) +Yg(Z,X) - Zg(X,Y)+ g(\pi[X,Y],Z)- g(\pi[Y,Z],X)- g(\pi[X,Z],Y),
\end{equation} where $X,Y,Z\in\Gamma(\D)$.
The curvature tensor of the Schouten connection $\nabla$ was defined by Schouten
in the following way,
\begin{equation}\label{Schouten}
R(X,Y)Z=\nabla_X\nabla_Y Z-\nabla_Y\nabla_X Z-\nabla_{\pi[X,Y]}Z-\pi\left([\pi' ([X,Y]),Z]\right),\quad \text{for all } X,Y,Z\in\D.
\end{equation}
It is straightforward to check that $R$ is a section of $\Lambda^2\D^*\otimes (\mathrm{End}(\D))$, but one can show more. 
Using that  $\pi'([X,Y])=-\d \theta(X,Y)\xi$, we get
\begin{equation}\label{Schouten2}
R(X,Y)Z=\nabla_X\nabla_Y Z-\nabla_Y\nabla_X Z-\nabla_{\pi[X,Y]}Z
+\d\theta(X,Y)[\xi,Z],
\end{equation}
so that with (\ref{metriccondit}) we have that 
\begin{equation}
\label{SCTsym}
g(R(X,Y)U,V)+g(R(X,Y)V,U)=-2\d\theta(X,Y) \bt(U,V).\end{equation}
This shows that the Schouten curvature tensor $R$ is in fact a section of $\Lambda^2\D^*\otimes \left(\so(\D)\+ \langle \bt\rangle \right)$, where we define the vector bundle
\[\so(\D)=\{A\in \mathrm{End}(\D) \mid g(AX,Y) +g(X,AY)=0\ \text{for all } X,Y\in \D\}.\]
It also satisfies the Bianchi identity 
\begin{equation}\label{BianchiSch}
\underset{X Y Z}{\Su} R(X,Y)Z=0,  \quad\text{for all } X,Y,Z\in\D,
\end{equation}
where $\underset{X Y Z}{\Su}$ denotes the  sum over cyclic permutations of $X, Y$ and $ Z$.

\subsection{Extended connections and the adapted connection} Let $\nabla $ be the Schouten connection on $(M,\theta,g)$.
An {\it extended connection} is a connection on the vector bundle $\D\to M$, 
$$\nabla^\D:\Gamma(\D)\longrightarrow  \Gamma(T^*M\otimes \D)$$ such that $$\nabla^\D_XY=\nabla_XY,\quad  \text{ for all $X,Y\in\Gamma(\D)$.}$$   
An arbitrary extended connection $\nabla^\D$  defines a unique  field $ \N$ of endomorphisms of $\D$ by 
\[\N(X)=\nabla^\D_\xi X-[\xi,X ].\] 
This is the analogue of the well-known {\em Nomizu operator}.
Hence,  for an endomorphism $\N $ of $\D$ we denote  by $\nabla^\N$ the  extended connection  
that is given by
\begin{equation}\label{nablaN} 
\nabla^{\N}_\xi X=[\xi,X]+\N(X),\quad \nabla^{\N}_XY=\nabla_XY,\quad\text{for all } X,Y\in\D,
\end{equation} 
and every extended connection is of this form.
The curvature tensor of the connection $\nabla^{\N}$ is given by
\begin{eqnarray}
\label{RNXY} R^\N (X,Y)&=&R(X,Y)+\d\theta(X,Y)\N ,\\
\label{RNxiX}
R^\N(\xi,X)&=&(\mathcal{L}_\xi\nabla)_X-\nabla_X\N,
\end{eqnarray}
and it holds that
\begin{equation}
\label{RxiXYYX}
R^\N(\xi,X)Y-R^\N(\xi, Y)X=-(\nabla_X\N)Y+(\nabla_Y\N) X,\end{equation}
where $X,Y,Z\in \D$ and $R$ is the Schouten curvature tensor. The first two identities follow from direct computation using~(\ref{projections}) and~(\ref{Schouten2}), see also~\cite{Papa2016}. 
The last equality follows from the Jacobi identity for vector fields and the identity~(\ref{Lpi}).

It is straightforward to check that an extended connection $\nabla^\N$ is compatible with the sub-Riemannian metric $g$, i.e.,~$\nabla^\N g=0$, if and only if the symmetric part of $\N$ is equal to $\frac12 \L_\xi g$, i.e.,
\begin{equation}\label{defA1} 
g(\N(X),Y)+g(X,\N(Y))=(\mathcal{L}_\xi g)(X,Y) =2\bt(X,Y),\quad \text{for all } X,Y\in\D.\end{equation}
This yields the definition of the 
 {\em adapted connection}, which is defined as the extended connection $\nabla^\bt$, i.e.,~the extended connection for which the Nomizu operator is given by \[\N=\bt=\tfrac12\L_\xi g,\] see, e.g., \cite{FG,FGR}. It is compatible with the metric by definition.

The curvature tensor $R^\bt$ of the adapted connection $\nabla^\bt$  is fully determined by $R$, $\theta$ and $\bt$ via~(\ref{RNXY}) and~the following.
\begin{lem}\label{lemRAxi} The curvature tensor of the adapted connection $\nabla^\bt$  satisfies
	$$g(R^\bt(\xi,X)Y,Z)=g((\nabla_Y\bt)X,Z)-g((\nabla_Z\bt)X,Y),\quad\text{for all } X,Y,Z\in\D.$$
	In particular,
	\begin{equation}\label{BianchiRxi}
\underset{X Y Z}{\Su} \ g(R^\bt(\xi ,X)Y,Z)=0,  \quad\text{for all } X,Y,Z\in\D.
\end{equation} 
	
\end{lem}

\begin{proof}
For any extended connection $\nabla^\N$ that is compatible with $g$, $R^\N(\xi,.)$ is a section of  $\D^*\otimes \so(\D)\cong \D^*\otimes \Lambda^2\D^*$, 
and so is the tensor field defined by
	$$P(X,Y,Z):=g(R^\N(\xi,X)Y,Z)-g((\nabla_Y\N)X,Z)+g((\nabla_Z\N)X,Y).$$
	Then from~(\ref{RxiXYYX}) we get that
	\[P(X,Y,Z)-P(Y,X,Z)=(\nabla_Z\N)(X,Y)- (\nabla_Z\N)(Y,X).\]
For the adapted connection $\N=\bt$ is symmetric, so that 
	 $P$ is in the kernel 
of the skew symmetrisation in the first two components
$ \D^*\otimes \so(\D) \to \Lambda^2\D^*\otimes \D$. This  however  is an isomorphism, so that  $P=0$ and the statement follows.
 \end{proof}
 We summarize  the algebraic properties of the curvature of the adapted connection. For this we define the  vector  bundle of algebraic curvature tensors
 \[
 \R(\so(\D))=
 \left\{Q\in \Lambda^2\D^*\otimes \so(\D)\mid 
\underset{X Y Z}{\Su}Q(X,Y)Z  =0,\   \text{for all } X,Y,Z\in\D\right\},\]
and the {\em weak curvature tensors}, see \cite{Leistner},
\[
 \P(\so(\D))=
 \left\{P\in \D^*\otimes \so(\D)\mid 
\underset{X Y Z}{\Su}g( P(X)Y,Z)  =0,\   \text{for all } X,Y,Z\in\D\right\}.\]
Since every $Q\in \R(\h)$ satisfies pairwise symmetry, $g( Q(X,Y)Z,U)=g( Q(Z,U)X,Y)$, we have a $\so(\D)$-invariant homomorphism that sends 
 $R\in\R(\so(\D))$ and  $X\in \D$ to  $R(X,.)\in \P(\so(\D))$. 
Returning to the curvature $R^\bt$ of the adapted connection, from Lemma~\ref{lemRAxi} we get that
\[P^\bt:=R^\bt(\xi,\cdot)\in \P(\so(\D)).\] On the other hand, for 
\[Q^\bt:=R^\bt |_{\Lambda^2\D} \in \Lambda^2\D^*\otimes \so(\D)\]
by~(\ref{RNXY}) we only get the following Bianchi identity
\begin{equation}\label{BianchiTWnew}
\underset{X Y Z}{\Su} Q^\bt(X,Y)Z=\underset{X Y Z}{\Su} \d\theta(X,Y)\bt (Z),  \quad\text{for all } X,Y,Z\in\D.
\end{equation}
Hence, instead of pairwise symmetry we have
\begin{equation}
\label{pairwise}
g(Q^\bt(X,Y)U,V)-g(Q^\bt (U,V)X,Y) =
\underset{U,V}{\Lambda}\left( \d\theta(Y,U)\bt(V,X)- 
\d\theta(X,U)\bt(V,Y) 
\right),
\end{equation}

where $\underset{X,Y}{\Lambda}$ denotes the skew symmetrisation of the tensor with respect to $X$ and $Y$ (without a normalizing factor of $1/2$). 
Indeed, using abstract indices, we have
\[
Q^\bt_{[abc]d}
+
Q^\bt_{[bcd]a}
+
Q^\bt_{[cda]b}
+
Q^\bt_{[dab]c}
=\tfrac{2}{3}\left( Q^\bt_{bdac}-Q^\bt_{acbd}\right)
\]
where the brackets denote skew symmetrisation (with normalizing factor).
On the other hand, by~(\ref{BianchiTWnew}), we have that
\[
Q^\bt_{[abc]d}
+
Q^\bt_{[bcd]a}
+
Q^\bt_{[cda]b}
+
Q^\bt_{[dab]c}
=
\tfrac{4}{3}\left( \d\theta_{a[b}\bt_{d]c}-\d\theta_{c[b}\bt_{d]a}\right),
\]
so that
\[
Q^\bt_{abcd}-Q^\bt_{cdab}=2 \left(\d\theta_{b[c}\bt_{d]a}- \d\theta_{a[c}\bt_{d]b}\right)
,\]
 which confirms formula~(\ref{pairwise}).

 \subsection{The Wagner connection}
Before we introduce another extended connection, we recall some algebraic preliminaries.
For one-forms $\sigma, \mu\in \D^*$ we define $\sigma\wedge \mu=\sigma\otimes \mu-\mu\otimes \sigma\in \Lambda^2\D^*$, and 
for vectors $X,Y\in\D$, we define the bivector
$X\wedge Y\in \Lambda^2 \D$ as 
\[X\wedge Y=X\otimes Y-Y\otimes X.\]
For a 2-form 
$\omega$ on $\D$ with values in a vector bundle over $M$, we set $$\omega(X\wedge Y)=2\omega(X,Y).$$
This allows us to consider the value $\omega(\alpha)$ of the 2-form $\omega$ on any bivector $\alpha\in \Lambda^2 \D$.

Given a non-degenerate $2$-form $\omega\in \Lambda^2\D^*$ we can define the bi-vector  $\omega^{-1}\in \Lambda^2\D$ by the following convention. If $(e_a)_{a=1}^{2m}$ is a basis for $\D$ at a point, we denote the dual basis of $\D^* $ by $e^a$, so that $\omega=\omega_{ab}e^a\wedge e^b$, and define 
\[\omega^{-1}=\hat\omega^{ab}e_a\wedge e_b,\quad\text{ where $\omega_{ac}\hat\omega^{cb}=\delta_a{}^b$.} \]
%
This implies that  \[\omega(\omega^{-1})=\omega_{ab}\hat\omega^{ab}=-2\mathrm{rk}(\D)=-4m.\] In particular, we have 
\[\Lambda^2\D=\ker(\omega)\+\langle \omega^{-1}\rangle.\]
With this notation, for any extended connection $\nabla^\N$ that is compatible with  $g$, we can compare its curvature to the  curvature tensor of the Schouten connection. From
 \eqref{RNXY}, we obtain
 \begin{align}R(\alpha)&=R^\N(\alpha)\in \so(\D),\quad  \text{ for all  $\alpha\in \Lambda^2\D$ such that $ \d\theta(\alpha)=0$},\\
 \label{preWagner}
R((\d\theta)^{-1})&=R^\N((\d\theta)^{-1})+4m\N.
\end{align}

\begin{prop}
There is a unique extended connection $\nabla^\N$ that is compatible with $g$ and whose curvature satisfies 
 \begin{align}
 \label{Wagner2}
 R^\N ((\d\theta)^{-1})&=0.
\end{align}
This connection is called {\em Wagner connection} after \cite{Wagner41}, and we denote it by $\nabla^W$.
\end{prop}
 \begin{proof}
For the existence of this connection, set
\begin{equation}\label{NWagner} \N=\N^W:=\tfrac{1}{4m}R((\d\theta)^{-1}),\end{equation}
and define the extended connection $\nabla^W:=\nabla^\N$.
From (\ref{SCTsym}) we get that
\[\N(X,Y)+\N(Y,X)= \tfrac{1}{4m}g\left(R((\d\theta)^{-1})X,Y)+ g(R((\d\theta)^{-1})Y,X\right)= - \tfrac{1}{2m}\d\theta(\theta^{-1})\bt=2\bt,
\]
so that $\nabla^W$ is compatible with the metric. From (\ref{preWagner}) we see that 
(\ref{Wagner2}) is satisfied. Since every extended connection is uniquely determined by $\N$, the statement follows.
\end{proof}
Let $\nabla^W$ be the Wagner connection defined by  $\N^W$ as in (\ref{NWagner}) and let
$$\N^W=\bt+C,$$ be the decomposition into the symmetric and skew-symmetric part, with 
\[C=\tfrac{1}{4m} R^\bt(\d\theta^{-1})=
\tfrac{1}{4m} R(\d\theta^{-1})- \bt.
\]
so that
\[\nabla^W_\xi X=\nabla^\bt_\xi(X)+C(X).\]
By (\ref{RNXY}) and (\ref{RNxiX}), the curvature of the Wagner and adapted connections are related as
\begin{align}
\label{RWtxy}
R^W(X,Y)&=R^\bt(X,Y)+\d\theta (X,Y)C,\\
R^W(\xi,X)&=R^\bt(\xi,X)-\nabla_XC,\label{RWRAC}
\end{align}
where $X,Y\in\D $. 
\begin{rem}
To be precise, in \cite{FG,FGR} the authors work with  affine connections, i.e.,~connections on $TM$ instead of  $\D$. Their affine connections 
$\overline{\nabla}^N$ can be obtained from our extended connection  $\nabla^N$ by requiring $\overline{\nabla}^N\xi=0$. The torsion of this connection is then given as 
\[T(X,Y)= \nabla^\bt_XY-\nabla^\bt_YX-[X,Y]=\d\theta(X,Y)\xi,\quad T(\xi,X)=N(X).\]
This applies in particular to $N=\bt$, hence the name {\em sub-torsion} for $\bt$. It also applies to the Wagner connection, and we will come back to this in Section~\ref{secLocsubsym}. Note there is a strong interest to holonomy of connections with torsion, see the survey \cite{A10}.

Note that in general the Wagner connection is different from the  {\em special connection} $\tilde\nabla$ in \cite[Section 2.2]{FGR}. The latter is defined by the following condition condition for its curvature $\tilde{R}$: if  $\alpha =\sum_i \lambda_iX_i\wedge Y_i$,  where $(X_i,Y_i)$ is a local orthonormal frame of $\D$ such that $\d\theta (X_i,Y_j)=\lambda_i\delta_{ij}$, then $\tilde{R}(\alpha)=0$.
\end{rem}

\section{Holonomy}\label{secHol}
In this section, let $(M,\theta,g)$ be a contact sub-Riemannian manifold and $\nabla$ be the Schouten connection. For extended connections we  use the notation of the previous section.
 A piece-wise smooth curve in $M$ is called {\it horizontal} if it is tangent to $\D$.

\subsection{Holonomy of the Schouten connection}
The Schouten connection $\nabla$ defines a parallel transport along any  horizontal curve $\gamma:[a,b]\to M$,  
\begin{equation}\label{tau}\tau_\gamma: \D_{\gamma(a)}\to \D_{\gamma(b)}.\end{equation} 
Let $x\in M$. The {\it (horizontal) holonomy group} of the Schouten connection $\nabla$ at the point $x\in M$ is the group of parallel transports along piece-wise smooth horizontal loops at the point $x\in M$. We denote this group by $\Hol_x(\nabla)$.  The holonomy group $\Hol_x(\nabla)$ is a Lie subgroup of $\Or (\D_x,g_x)$.

The {\it restricted holonomy group} $\Hol^0_x(\nabla)$ is the subgroup of 
$\Hol_x(\nabla)$ corresponding to contractable loops. The restricted holonomy group $\Hol^0_x(\nabla)$ is the identity component of $\Hol_x(\nabla)$, \cite{CGJK,FGR,HMCK}, and 
the corresponding Lie subalgebra of $\so (\D_x)$ is called {\it the holonomy algebra} and is denoted by $\hol_x(\nabla)$. The parallel transport along a horizontal curve defines the isomorphism of the holonomy groups and algebras at the initial and the endpoint of the curve. Since $\D$ is a contact distribution, the Chow--Rashevskii theorem ensures that any two points in $M$ are connected by a horizontal curve, so the holonomy is independent of the point up to conjugation in $\mathrm{GL}(n,\Real)$.

The following are versions of the Ambrose--Singer holonomy Theorem for the horizontal connection.
For this define 
\[\ker(\d\theta):=\{\alpha \in \Lambda^2 \D\mid \d\theta(\alpha)=0\},\]
the space of Legendrian bi-vectors.  
First we have  \cite[Corollary 2.4]{FGR}, see also \cite{CGJK}, the following.

\begin{theorem}\label{ThAS1} The Lie algebra $\hol_x(\nabla)$ is spanned by 
	the following endomorphisms of $\D_x$,
	$$\tau_\gamma^{-1}\circ R(\alpha)\circ\tau_\gamma,$$
	where $\gamma:[a,b]\to M$ is an arbitrary piece-wise smooth horizontal curve starting at $x$ and $\alpha\in\ker(\d\theta)_{\gamma(b)}$.
\end{theorem}

We also have the following \cite[Theorem~2.16]{CGJK}, \cite[Theorem~2.1]{FGR}.

\begin{theorem} Let $(M,\theta,g)$ be a contact sub-Riemannian manifold and let $\nabla^N$ be an extension of the Schouten connection $\nabla$. Suppose that there is a 
	  bi-vector field $\alpha\in\Gamma(\Lambda^2 \D)$ such that the function $\d\theta(\alpha)$ is non-vanishing and it holds $R^N(\alpha)=0$. Then, $\Hol(\nabla)=\Hol(\nabla^N)$.   \end{theorem}

Since $\alpha=\d\theta^{-1}$ satisfies the assumption of this theorem and by the definition of the Wagner connection, we can conclude the following.

\begin{theorem}\label{thholW} Let $(M,\theta,g)$ be a contact sub-Riemannian manifold. Then the horizontal holonomy group $\Hol(\nabla)$ coincides with the holonomy group $\Hol(\nabla^W)$ of the Wagner connection. 
\end{theorem}

\subsection{The connection $\nabla^\theta$}\label{secConTheta}

Let $(M,\theta)$ be a contact manifold.  
Using the 1-form $\theta$, we define the connection $\nabla^\theta$ on the trivial line bundle $\left<\xi\right>$ over $M$ by setting 
$$\nabla^\theta_X\xi=\theta(X)\xi,\quad X\in\Gamma(TM).$$
The curvature of the connection $\nabla^\theta$ coincides with $\d\theta$, i.e.,
$$R^\theta(X,Y)\xi=\d\theta(X,Y)\xi,\quad X,Y\in TM.$$
The parallel transport along a closed curve $\mu(t)$ at a point $x$ is given by
\begin{equation}
\label{parnabtheta}
\xi_x\mapsto\exp\left(-\int_\mu \theta\right)\xi_x,
\end{equation}
see, e.g.,~\cite{GalHolK}.
The holonomy group $\Hol_x(\nabla^\theta)$ is isomorphic to the subgroup $\Real_+\subset\GL(1,\Real)$. 


\subsection{The group $\widetilde{\Hol}_x(\nabla^N)$}
\label{tildehol-sec}

Let $\nabla^N$ be a metric extended connection, $R^\N$ its curvature tensor, and $\tau^\N_\gamma$ the corresponding parallel transport along a piece-wise smooth curve $\gamma:[a,b]\to M$.
Denote by $\widetilde{\Hol}_x(\nabla^N)$ the subgroup of the holonomy group   ${\Hol}_x(\nabla^N)$ of $\nabla^N$ that consists of parallel transports for the connection $\nabla^N$ along the loops in $M$ that satisfy the condition
\begin{equation}\label{condintegr0}\int_\mu\theta=0.\end{equation}
It holds 
$${\Hol}_x(\nabla)\subseteq\widetilde{\Hol}_x(\nabla^N)\subseteq {\Hol}_x(\nabla^N) \subseteq \Or(\D_x,g_x).$$ 

\begin{theorem} \label{theorHoltild}
Let $(M,g,\theta)$ be a sub-Riemannian contact structure and let $\nabla^N$ be a metric extended connection.
Then	 it holds that either $\widetilde{\Hol}_x(\nabla^N)={\Hol}_x(\nabla^N)$, or $\widetilde{\Hol}_x(\nabla^N)$ is a co-dimension one  normal Lie subgroup of  ${\Hol}_x(\nabla^N)$.

		The Lie algebra $\widetilde{\hol}_x(\nabla^N)$  of the Lie group $\widetilde{\Hol}_x(\nabla^N)$  is spanned by the endomorphisms
	$$(\tau^N_\gamma)^{-1}\circ R^N(\alpha)\circ\tau^N_\gamma,\quad (\tau_\gamma^N)^{-1}\circ R^N(\xi,X)\circ\tau^N_\gamma, $$
	where $\gamma:[a,b]\to M$ is an arbitrary piece-wise smooth curve starting at $x$,  $X\in \D_{\gamma(b)} $ and $\alpha\in \ker(\d\theta)_{\gamma(b)}$.
\end{theorem}


\begin{proof} 
Consider the affine connection $\tnab^N:=\nabla^{N}\oplus\nabla^\theta$ on the tangent bundle $TM=D\oplus\left<\xi\right>$ defined by
$$\tnab^N(Y+\xi)=\nabla^{N}Y+\nabla^\theta\xi,\qquad \text{ for }Y\in\Gamma(\D).$$ The holonomy group 
${\Hol}_x(\tnab^\N)$ of this connection consists of the pairs $(\tau^N_\gamma,\tau^\theta_\gamma)$, where $\gamma$ is a loop at $x$, i.e., ${\Hol}_x(\tnab^\N)$ is contained in the product 
\[{\Hol}_x(\tnab^\N)\ \subseteq\ {\Hol}_x(\nabla^N)\times {\Hol}_x(\nabla^\theta)={\Hol}_x(\nabla^N)\times \Real_{+}.\] 
Consider the projection
$$\pr_2:  {\Hol}_x(\nabla^N)\times \Real_+\longrightarrow \Real_+,$$
and its restriction
$$\pi_2:=\pr_2|_{{\Hol}_x(\tnab^\N)}:  {\Hol}_x(\tnab^\N)\longrightarrow \Real_+,$$
which are  Lie group homomorphisms.   Moreover, we have  Lie group homomorphism $\iota:\Hol(\nabla^N)\ni A\to (A,1)\in \Hol(\nabla^N)\times \Real_{+}$ with the exact sequence of Lie groups 
\[\begin{array}{ccccccccccc}
\1&\longrightarrow& \Hol(\nabla^N)&\stackrel{\iota}{\longrightarrow} & \Hol(\nabla^N) \times \Real_+ &\stackrel{\pr_2}{\longrightarrow}& \Real_+ &\longrightarrow&\1.
\end{array}
\]
By the definition of $\widetilde{\Hol}_x(\nabla^N)\subseteq \Hol_x(\nabla^N) $ and~(\ref{parnabtheta}), and by the fact that $\nabla^\theta$ is not flat,  it restricts to an exact sequence of groups
\[\begin{array}{ccccccccccc}
\1&\longrightarrow&\widetilde{\Hol}_x(\nabla^N) &\stackrel{\iota}{\longrightarrow} & \Hol(\tnab^N) &\stackrel{\pi_2}{\longrightarrow}& \Real_+ &\longrightarrow&\1.
\end{array}
\]
Hence, $\ker (\pi_2)= \widetilde{\Hol}_x(\nabla^N) \times \{1\}$ is a normal Lie subgroup of ${\Hol}_x(\tnab^\N)$ of  co-dimension $1$.
Moreover it shows 
 that $\widetilde{\Hol}_x(\nabla^N)$ is a
normal Lie subgroup of 
$ {\Hol}_x(\nabla^\N)$, of co-dimension $0$ or $1$, depending on whether $\Hol(\tnab^N)= \Hol(\nabla^N) \times \Real_+ $ or not.
See also  \cite[Th. 5]{GalHolK}.

%
%
%
%

By the Ambrose--Singer Theorem, the Lie algebra  ${\hol}_x(\tnab^\N)$ is spanned by the endomorphisms of the form
$$\big((\tau^N_\gamma)^{-1}\circ R^N(\beta)\circ\tau^N_\gamma,(\tau^\theta_\gamma)^{-1}\circ R^\theta(\beta)\circ\tau^\theta_\gamma\big),$$
where $\gamma:[a,b]\to M$ is an arbitrary piece-wise smooth curve starting at $\gamma(a)=x$, and $\beta\in \Lambda^2T_{\gamma(b)}M$. Since 
$$\widetilde{\hol}_x(\nabla^N)\cong\ker (\d\pi_2),$$
the above endomorphism belongs to $\widetilde{\hol}_x(\nabla^N)$ if and only if $R^\theta(\beta)=0$, which means that $\d\theta(\beta)=0$, i.e.,~$\beta=\xi\wedge X+\alpha$, where $X\in \D_{\gamma(b)}$,  $\alpha\in\ker(\d\theta)_{\gamma(b)}$. This proves the last statement of the theorem. \end{proof}

\subsection{The group $\widehat{\Hol}_x(\nabla^N)$.} 
Let us assume that the Reeb vector field $\xi$ is complete 
and we denote its flow by $\varphi_t$. Fix a point $x\in M$. If the orbit $\varphi_t(x)$ of the point $x$ is cyclic, then we denote by $t_x$ the smallest positive $t_x$ such that $\varphi_{t_x}(x)=x$. Otherwise we assume that $t_x=0$.
Let $$\mu:[a,b]\to M$$ be  a piece-wise smooth curve. In \cite{GalHolK} it is shown that the  curve   

\begin{equation}\label{deftildmu}
\tilde \mu(t)=\varphi_{f(t)}(\mu(t)),\quad\text{ 
where}\quad f(t)=-\int_a^t\theta(\dot\mu(r))\d r,\quad  t\in[a,b],\end{equation}
 is horizontal and $\tilde\mu(a)=\mu(a)$. Moreover, if $\mu$ is a loop at $x$, then $\tilde\mu$ is a loop  if and only if 
 	\begin{equation}\label{condintegr}\int_\mu\theta\in \mathbb{Z}t_x.\end{equation}
Let $\mu$ be any curve starting at $x$ and satisfying \eqref{condintegr0}.
	Let $\tilde\mu$ be the corresponding horizontal curve defined in~(\ref{deftildmu}), where condition~\eqref{condintegr0} ensures that the endpoints of $\mu$ and $\tilde \mu$ coincide, so that $\mu*\tilde\mu^{-1}$ is a loop.  Let $\widehat{\Hol}_x(\nabla^N)$ be the group generated by parallel transports along all loops of the form $\mu*\tilde\mu^{-1}$. 
It is clear that $\widehat{\Hol}_x(\nabla^N)\subseteq \widetilde{\Hol}_x(\nabla^N)$ is a subgroup. 

Let $\mu:[a,b]\to M$ be any curve starting at $x$ and satisfying \eqref{condintegr0}.
Consider the $s$-family $\mu_s$, $s\in[0,1]$, of the curves 
 \begin{equation}
\label{mus}
t\longmapsto \mu_s(t)=\varphi_{sf(t)}(\mu(t)),\end{equation}
with $\mu_0=\mu$, and $\mu_1=\tilde\mu$.

\begin{lem}
\label{homotopylem}
	For each $s\in[0,1]$, the curve $\mu_s$ satisfies \eqref{condintegr0}.
\end{lem} 

\begin{proof} It holds $$\dot\mu_s(t)=s\dot f(t) \xi_{\varphi_{sf(t)}\mu(t)}+(d\varphi_{sf(t)})\dot\mu(t),$$ where 
$$d\varphi_{sf(t)}:T_{\mu(t)}M\to T_{\varphi_{sf(t)}\mu(t)}M$$ is the differential of the diffeomorphism $\varphi_{sf(t)}$. Since $\varphi$ preserves $\theta$, we get 
$$\theta(\dot\mu_s(t))=s\dot f(t)+\theta ((d\varphi_{sf(t)})\dot\mu(t))=s\dot f(st)+\theta (\dot\mu(t)).$$ Now, 
$$\int_{\mu_s}\theta=s(f(b)-f(a))+\int_{\mu}\theta=0.$$ \end{proof}

\begin{prop} 
Let $(M,g,\theta)$ be a sub-Riemannian contact structure with complete Reeb vector field and let $\nabla^N$ be an extended connection.
The group $\widehat{\Hol}_x(\nabla^N)$ is a connected Lie subgroup of $\widetilde{\Hol}_x(\nabla^N)$.
\end{prop}

\begin{proof}
Let $\mu$ be a curve satisfying~(\ref{condintegr0}) and let $\tau^N_{\mu\ast \tilde{\mu}^{-1}}$ be an arbitrary element in $\widehat{\Hol}_x(\nabla^N)$.
By Lemma~\ref{homotopylem} the curve
\[\tau_s^N:=\tau^N_{\mu_s*\tilde\mu_s^{-1}},\]
where $\mu_s$ is defined in~(\ref{mus}), is a curve in $\widehat{\Hol}_x(\nabla^N)$ and we claim that it joins 
$\tau^N_{\mu\ast \tilde{\mu}^{-1}}$ with the identity. Indeed, $\mu_0*\tilde\mu_0^{-1}= \mu*\tilde\mu^{-1}$ and, since $\tilde \mu $ is horizontal so that $\tilde{\tilde{\mu}}=\tilde\mu$, we have
\[\mu_1*\tilde\mu_1^{-1}= \tilde\mu*\tilde \mu^{-1}.\] 
Hence $\tau_1^N= \tau^N_{\tilde\mu*\tilde \mu^{-1}}=\id_{\D_x}$, so that 
$\widehat{\Hol}_x(\nabla^N)\subseteq \widetilde{\Hol}_x(\nabla^N)$ is a connected  subgroup, which implies that it is a Lie subgroup. \end{proof}
%
%
%

\begin{prop}\label{product-prop}
Let $(M,g,\theta)$ be a sub-Riemannian contact structure with complete Reeb vector field and let $\nabla^N$ be a metric extended connection.
The Lie group $\widetilde{\Hol}_x(\nabla^N)$ is generated by the subgroups $\widehat{\Hol}_x(\nabla^N)$ and $\Hol_x(\nabla)$, i.e. 
\[
\widetilde{\Hol}_x(\nabla^N)= \widehat{\Hol}_x(\nabla^N)\cdot\Hol_x(\nabla).\]
Moreover, 
 $\Hol_x(\nabla)$ normalizes  $\widehat{\Hol}_x(\nabla^N)$.
\end{prop}

\begin{proof}

If a loop $\gamma$ satisfies \eqref{condintegr0}, then
$$\tau^N_\gamma=\tau^N_{\gamma*\tilde\gamma^{-1}}\circ\tau^N_{\tilde\gamma},$$ which implies the first statement.

To show the normalizing condition, let $\mu$ be a curve starting at $x$ and satisfying \eqref{condintegr0}, and $\gamma$ is a horizontal loop at $x$. Then 
$$(\tau^N_\gamma)^{-1}\circ\tau^N_{\mu*\tilde\mu^{-1}}\circ\tau^N_\gamma=(\tau^N_\gamma)^{-1}\circ\tau^N_\mu\circ(\tau^N_\gamma)\circ(\tau^N_\gamma)^{-1}\circ(\tau^N_{\tilde\mu})^{-1}\circ\tau^N_\gamma=\tau^N_{\gamma*\mu*\gamma^{-1}}\circ
(\tau^N_{\gamma*\tilde\mu*\gamma^{-1}})^{-1}.$$  Since $\gamma$ is horizontal, it holds $\tilde\gamma=\gamma$, and
$$\widetilde{\gamma*\mu*\gamma^{-1}}=\tilde\gamma*\tilde\mu*\tilde\gamma^{-1}=
\gamma*\tilde\mu*\gamma^{-1}.$$
This proves the proposition. \end{proof}


\begin{prop}\label{propHhatistrivial} 
Let $(M,g,\theta)$ be a sub-Riemannian contact structure with complete Reeb vector field and let $\nabla^N$ be a metric extended connection.
If $R^N(\xi,\cdot)=0$, then the group $\widehat{\Hol}_x(\nabla^N)$ is trivial.
\end{prop}

\begin{proof}
Let $\mu:[a,b]\to M$ be a curve satisfying \eqref{condintegr0} and with $\mu(a)=x$. Let $W=[0,1]\times [a,b]$. Consider the map
$$F:W\to M,$$
$$F(s,t)=\varphi_{sf(t)}(\mu(t)).$$
Consider the pull-back bundle $F^*\D\to W$ and the induced connection $\nabla^F$ along the map $F$. It holds $$\frac{\d}{\d s}F(s,t)=f(t)\cdot\xi_{F(s,t)}.$$
This implies that the curvature of the induced connect $\nabla^F$ is determined by 
$R^N(\xi,\cdot)$, which by assumption vanishes, so that the induced connection is flat and the parallel transport along the curve $\mu*\tilde\mu^{-1}$ is trivial. \end{proof}


\subsection{A Codazzi condition}

We will say that a
a sub-Riemannian contact structure is {\em Codazzi} if the sub-torsion $\bt$ satisfies the Codazzi equation with respect to the Schouten connection, i.e.,
\[(\nabla_X\bt)Y=(\nabla_Y\bt )X,\quad\text{ for all $X,Y\in \D$.}\]
From Lemma \ref{lemRAxi} we have the following.

\begin{cor} A  a contact sub-Riemannian manifold  $(M,\theta,g)$ is Codazzi if and only if  the curvature tensor $R^\bt$ of the adapted connection $\nabla^\bt$ satisfies 
	$R^\bt(\xi,\cdot)=0$.
\end{cor}

From Propositions~\ref{product-prop} and~\ref{propHhatistrivial} we obtain the following result.

\begin{cor} Let $(M,\theta,g)$ be a contact sub-Riemannian manifold that is Codazzi and such that the Reeb vector field $\xi$ is complete. Then it holds
$\Hol_x(\nabla)=\widetilde{\Hol}_x(\nabla^\bt)$.
\end{cor}

This and Theorem \ref{theorHoltild} imply the following.

\begin{cor}\label{cordech}
Let $(M,\theta,g)$ be a contact sub-Riemannian manifold that is Codazzi and such that the Reeb vector field $\xi$ is complete. Then either ${\Hol}_x(\nabla)={\Hol}_x(\nabla^\bt)$, or ${\Hol}_x(\nabla)$ is a co-dimension one  normal subgroup of  ${\Hol}_x(\nabla^\bt)$.
\end{cor}
For the final result of this section we introduce the notion of (weak) curvature tensors. 
For a subalgebra $\h\subseteq\so(l)$, the space $\P(\h)$ of {\em weak curvature tensors}, as defined in \cite{Leistner}, consists of linear maps from $\Real^l$ to $\h$ satisfying the identity
$$\underset{X Y Z}{\Su}\langle P(X)Y,Z\rangle =0,  \quad\text{for all } X,Y,Z\in\Real^l.$$
Let $\R(\h)$ 
denote the space of algebraic curvature tensors, i.e.,~of elements in $\Lambda^2 (\Real^l)^*\otimes \h$ that satisfy Bianchi-identity, 
$$\underset{X Y Z}{\Su}R(X,Y)Z=0,  \quad\text{for all } X,Y,Z\in\Real^l.$$
Since every $R\in \R(\h)$ satisfies pairwise symmetry, $\langle R(X,Y)Z,U\rangle=\langle R(Z,U)X,Y\rangle$, for
any  $R\in\R(\h)$ and  $X\in \Real^l$, we have $\R(X,.)\in \P(\h)$. 
From Lemma~\ref{lemRAxi} it follows that $R^\bt(\xi,\cdot)\in \P(\hol(\nabla^\bt))$, however 
$R^\bt |_{\Lambda^2\D}$ does not necessaily satisfy the Bianchi identity and hence is  not necessarily in $\R(\h)$. Nevertheless, we get the following.

\begin{prop}\label{propRieamHolor} Let $(M,\theta,g)$ be a contact sub-Riemannian manifold that is not Codazzi. Suppose that the holonomy algebra $\hol(\nabla^\bt)\subseteq\so(2m)$ of the adapted connection is irreducible. Then either 
$\hol(\nabla^\bt)\subseteq\so(2m)$ is the holonomy algebra of a Riemannian manifold or
$$\hol(\nabla^\bt)=\sp(k)\oplus\u(1),\quad m=2k.$$
\end{prop}

\begin{proof} 
From Lemma~\ref{lemRAxi} and the assumption that  the sub-Riemannian structure is not Codazzi  we have $0\not= R^\bt(\xi,\cdot)\in \P(\hol(\nabla^\bt))$. Let $\h\subseteq \hol(\nabla^\bt)$ be the subalgebra generated by images of elements from $\mathcal P(\hol(\nabla^\bt))$. The main result form \cite{Leistner} states that $\h\subseteq\so(2m)$ is the holonomy algebra of a Riemannian manifold. This implies that $\R(\h)\neq 0$ and hence $\R(\hol(\nabla^\bt))\neq 0$. Then the statement of the proposition follows from~\cite[Theorem~4.6]{C-S}. \end{proof}
From the proof we see that the result also holds when the assumption of non-Codazzi is replaced by the assumption $\R(\hol(\nabla^\bt))\not=0$. Then it follows directly from~\cite[Theorem~4.6]{C-S}.

\subsection{A further analysis of the Codazzi case}\label{secCodazziCase} In this section we generalize Corollary \ref{cordech} relaxing   the completeness assumption  for $\xi$.

\begin{theorem}\label{thHolnabl}
	Let $(M,\theta,g)$ be a contact sub-Riemannian manifold. Suppose that $\bt$ is Codazzi.  Then one of following conditions holds:
	\begin{itemize}
		\item[1.] $\Hol_x(\nabla)=\Hol_x(\nabla^\bt)$;
		\item[2.]  $\Hol_x(\nabla)\subset\Hol_x(\nabla^\bt)$ is a normal subgroup of co-dimension one.
	\end{itemize} 
\end{theorem}

\begin{proof}  Since the subbundle $\hol(\nabla)\subseteq\hol(\nabla^\bt)$ is $\nabla$-parallel,   $\nabla$ preserves the decomposition
$$
\hol(\nabla^\bt)=\hol(\nabla) \oplus \hol(\nabla)^\bot.$$ 
Let $C$ be the skew part of $N^W$ for the Wagner connection $\nabla^W$.  Since $C=\tfrac{1}{4m} R^\bt(\d \theta^{-1})$, we have that $C\in \hol(\nabla^\bt) $. Let $$C=C^\nabla+C^\bot$$ be the corresponding decomposition of the skew-symmetric part $C$ of $N^W$.
For the following recall that from Theorem~\ref{thholW} we have that  \begin{equation}
\label{holWhol}\Hol(\nabla^W)=\Hol(\nabla).\end{equation}

\begin{lem} It holds that $$\nabla^\bt C^\bot=0,$$ and
	$$\hol(\nabla^\bt)=\hol(\nabla) \oplus \left<C^\bot\right>  $$ is the direct sum of ideals.
\end{lem}

\begin{proof}  Since $\bt$ is Codazzi, it holds $R^\bt(\xi,,X)=0$, so that 
$$R^W(\xi,X)=-\nabla_XC=-\nabla_XC^\nabla-\nabla_XC^\bot,\quad\text{for all } X\in\Gamma(\D).$$
Since $R^W(\xi,X)$ takes values in $\hol(\nabla)$, it holds 
$$\nabla C^\bot=0.$$ This implies that $ C^\perp$ commutes with $\Hol(\nabla)$, and with~(\ref{holWhol})  yields  $\nabla^W C^\bot=0$.
We conclude that
$$\nabla^\bt_\xi C^\bot=(\nabla^W_\xi-C^\nabla-C^\bot)\cdot C^\bot=0.$$
Next, since
$$\nabla^\bt_\xi =\nabla^W_\xi-C^\nabla-C^\bot,$$
from~(\ref{holWhol}) we see that $\nabla^\bt_\xi$ preserves $\Gamma(\hol(\nabla))$. This implies that the bundle $\hol(\nabla)$ is $\nabla^\bt$-parallel.
For each point $y\in M$ and vectors $X,Y\in C_x$ it holds 
$$R^\bt_y(X,Y)=R^W_y(X,Y)+(d\theta)_y(X,Y)C_y\in\hol_y(\nabla)\oplus\Real (C^\bot)_y.$$
We conclude that for each curve $\gamma$ from $x$ to $y$ it holds
$$(\tau^\bt_\gamma)^{-1}\circ R^\bt_y(X,Y)\circ\tau^\bt_\gamma\in  \hol_x(\nabla)\oplus\Real (C^\bot)_x.$$ Note also that $R^\bt(\xi,\cdot)=0.$
Thus, by the Ambrose--Singer Theorem, 
$$\hol_x(\nabla^\bt)=\hol_x(\nabla) \oplus \left<C_x^\bot\right>.$$ \end{proof}

Note that the statement of the above Lemma also holds  in the case $C^\bot=0$.

\begin{lem}\label{lemxi1} Let $\lambda:[0,r]\to M$ be an integral curve of the vector field $\xi$. Then it holds
	$$\tau_\lambda^\bt=e^{r\cdot C^\bot_{\lambda(r)}}\circ \tau_\lambda^W\circ B(r),$$ where $B(r)\in\Hol_{\lambda(0)}(\nabla)$.
\end{lem}

\begin{proof}  For each $s\in [0,r]$, let 
$$B(s)=((\tau_\lambda^W)_s)^{-1}\circ e^{-s\cdot C^\bot_{\lambda(s)}}\circ (\tau_\lambda^\bt)_s:\D_{\lambda(0)}\to \D_{\lambda(0)},$$
be a curve in $\GL(\D_{\lambda(0)})$,
where $$(\tau_\lambda^\bt)_s,(\tau_\lambda^W)_s:\D_{\lambda(0)}\to\D_{\lambda(s)}$$ are parallel transports along the curve $\lambda|_{[0,s]}$.
Let $X(s)$ be a $\nabla^\bt$-parallel section of $\D$ along the curve $\lambda$, i.e., it holds
$$\nabla^\bt_{\dot\lambda(s)}X(s)=0,\quad (\tau_\lambda^\bt)_s X(0)=X(s). $$
Then, using that   $\nabla^W C^\perp=0$, we have 
\begin{eqnarray*}
0&=&\nabla_{\dot\lambda(s)}^\bt X(s)
\\
&=&\left(\nabla^W_{\dot\lambda(s)}-C_{\lambda(s)} \right) \left( e^{s\cdot C^\bot_{\lambda(s)}}\circ (\tau_\lambda^W)_s\circ B(s)\right) X(0)
\\
&=& \left(C^\bot_{\lambda(s)}\circ e^{s\cdot C^\bot_{\lambda(s)}}\circ (\tau_\lambda^W)_s\circ B(s)\right) X(0)+\left( e^{s\cdot C^\bot_{\lambda(s)}}\circ (\tau_\lambda^W)_s\circ B'(s)\right) X(0)
\\
&&{}-\left(C_{\lambda(s)}\circ e^{s\cdot C^\bot_{\lambda(s)}}\circ (\tau_\lambda^W)_s\circ B(s)\right) X(0).
\end{eqnarray*}
With  $[C,C^\perp]=0$, this implies 
$$(\tau_\lambda^W)_s\circ B'(s)-C^\nabla_{\lambda(s)}\circ (\tau_\lambda^W)_s\circ B(s)=0,$$
so that
$$ B'(s)\circ B^{-1}(s)=((\tau_\lambda^W)_s)^{-1}\circ C^\nabla_{\lambda(s)}\circ (\tau_\lambda^W)_s.$$ 
Since $C^\nabla\in \hol(\nabla)$, this  shows that $B(s)$ is the development of the curve 
$((\tau_\lambda^W)_s)^{-1}\circ C^\nabla_{\lambda(s)}\circ (\tau_\lambda^W)_s\in\hol_{\gamma(0)}(\nabla)$. \end{proof}

\begin{lem}\label{lemxi2} Let $\mu:[a,b]\to M$ be an arbitrary curve. Then there exist a horizontal curve $\bar\mu:[a,b]\to M$ such that
	$$\bar\mu(a)=\mu(a),\quad \bar\mu(b)=\mu(b),\quad \tau_{\bar\mu}=\tau^W_\mu.$$ 
\end{lem}

\begin{proof}  Let $\check\mu$ be a horizontal curve connecting the points $\mu(a)$ and $\mu(b)$. We obtain the loop $\mu*\check\mu^{-1}$ at the point $\mu(a)$.
Since the holonomy groups of the connections $\nabla$ and $\nabla^W$ coincide, there exists a horizontal loop $\hat\mu$ at the point $\mu(a)$ such that
$$\tau^W_{\mu*\check\mu^{-1}}=\tau_{\hat\mu}.$$ This implies that 
$\tau^W_{\mu}=\tau_{\hat\mu*\check\mu}$.
\end{proof}

Let $\gamma:[a,b]\to M$ be an arbitrary loop at the point $x$.
Since $\gamma([a,b])\subseteq M$ is compact, there exist
numbers $$a_0=a< a_1<\dots< a_{k-1}<a_k=b$$ such that,
for each restriction
$$\gamma_i=\gamma|_{[a_{i-1},a_i]},$$ the horizontal curve $\tilde \gamma_i$ defined by \eqref{deftildmu} is well-defined. 
By the definition of $\tilde\gamma_i$, there exists an integral curve $\lambda_i(s)$, $s\in [0,r_i]$ of the vector field $\xi$ connecting the end-points $\tilde\gamma(a_i)$ of the curve $\tilde\gamma_i$ with the end-point $\gamma(a_i)$ of the curve $\gamma_i$, and it holds 
$$r_i= \int_{\gamma_i}\theta.$$
Since $R^\bt(\xi,\cdot)=0$,  
the arguments from the proof of Proposition \ref{propHhatistrivial} imply
$$\tau^\bt_{\gamma_i}=\tau^\bt_{\lambda_i}\circ \tau_{\tilde\gamma_i}.$$
From Lemmas \ref{lemxi1} and \ref{lemxi2} it follows that there exists a horizontal curve $\mu_i$ connecting the points $\gamma(a_{i-1})$ and $\gamma(a_i)$ such that 
$$\tau^\bt_{\gamma_i}=e^{r_iC^\bot_{\gamma(a_i)}}\circ\tau_{\mu_i}.$$
Using this, the equality
$$\tau^\bt_\gamma=\tau^\bt_{\gamma_k}\circ\cdots\circ\tau^\bt_{\gamma_1},$$
and applying repeatedly the fact that $C^\bot$ is $\nabla$-parallel, we get 
$$\tau^\bt_\gamma=e^{(r_1+\cdots +r_k)C^\bot_x}\circ\tau_{\mu},$$
where $\mu$ is a horizontal loop at the point $x$.
It is obvious that 

$$r_1+\cdots +r_k=\int_\gamma\theta.$$
Thus, $$\tau^\bt_\gamma=\exp\left(\left(\int_\gamma\theta\right)\cdot {C^\bot_x}\right)\circ\tau_{\mu}.$$
Since
$C^\perp$ commutes with $\Hol(\nabla)$, this implies that the subgroup $\Hol_x(\nabla)\subseteq\Hol_x(\nabla^\bt)$ is normal. It is obvious that the map
$$\lambda:\Real_+\to\Hol_x(\nabla^\bt)/\Hol_x(\nabla),$$
$$\lambda:\exp\left(\int_\gamma\theta\right)\mapsto
\exp\left(\left(\int_\gamma\theta\right)\cdot {C^\bot_x}\right)\cdot \Hol_x(\nabla)$$
is well-defined and surjective. \end{proof}

\subsection{Holonomy  of K-contact sub-Riemannian manifolds}\label{secbt0}
In this section we will recall some results from \cite{GalHolK}, where $\xi$ was assumed to be complete. Thanks to Section \ref{secCodazziCase}, for these results we may drop the completeness assumption for $\xi$.
Recall that a contact sub-Riemannian manifold $(M,\theta,g)$ is called {\em K-contact} if it is sub-torsion-free, i.e.~$\bt=0$. 
Then  with Section~\ref{secCodazziCase} the results from \cite[Sections~7 and 8]{GalHolK} can be restated as in the following three theorems.

\begin{theorem}
	Let $(M,\theta,g)$ be a  K-contact sub-Riemannian manifold.  
	Then $\hol(\nabla^\bt)$ is the holonomy algebra of a Riemannian manifold. Moreover, it holds that either $\Hol(\nabla)=\Hol(\nabla^\bt)$ or $\Hol(\nabla)\subseteq\Hol(\nabla^\bt)$ is a normal subgroup of co-dimension one.
\end{theorem}

Since $\hol_x(\nabla^\bt)\subseteq \so(\D_x)$,
there exists an $\hol_x(\nabla^\bt)$-invariant orthogonal decomposition of the vector space $\D_x$,
\begin{equation}\label{deRD} \D_x=\D^0_x\oplus \D_x^1\oplus\cdots\oplus \D_x^r.\end{equation}
Moreover, since $\hol_x(\nabla^\bt)$ is a Riemannian holonomy algebra, it decomposes into the direct sum of ideals \begin{equation}\label{deRh}\hol_x(\nabla^\bt)=\h_1\oplus\cdots\oplus \h_r\end{equation}
such that $\h_i\subseteq \so(\D_x^i)$ is an irreducible Riemannian holonomy algebra for $i=1,\dots, r$ that acts trivially on $\D_x^j$ when $i\not=j$. In particular,   $\hol_x(\nabla^\bt)$ acts trivially on  $\D^0_x$.

\begin{theorem}\label{thCriter} Let $(M,\theta,g)$ be a simply connected K-contact sub-Riemannian manifold. Then the following conditions are equivalent: 
	\begin{enumerate}
		\item	
		The horizontal holonomy algebra  $\hol(\nabla)$ is a co-dimension one  ideal of  $\hol(\nabla^\bt)$.

				\item  The following conditions are satisfied:
				\begin{enumerate}[(i)]
				\item  there exists a $\nabla^\bt$-parallel  complex structure  $J$ on $\D$, i.e., $\hol(\nabla^\bt)\subseteq\u(m)$,
				\item the space $\D^0_x$ from decomposition \eqref{deRD} is trivial,
				\item each algebra $\h_i$ from  decomposition \eqref{deRh} is contained in $\u(\D^i_x)$ and it contains  $\u(1)=\Real J_x^i$, where $J_x^i={J_x}|_{\D^i_x}$, and
				\item
	we have
		\begin{equation}\label{conddtheta=}d\theta=b_1\rho^1+\cdots+ b_r\rho^r,\quad b_1,\dots, b_r\in \Real\backslash\{0\}.\end{equation}
		\end{enumerate}

	\end{enumerate}	
\end{theorem}

\begin{theorem}\label{thholnablaKcontact} 
	Suppose that the horizontal holonomy algebra  $\hol(\nabla)$ is a co-dimension one  ideal of  $\hol(\nabla^\bt)$. Let $J$ be a fixed transverse K\"ahler structure. Then there are decompositions \begin{align}\label{deRD1}  \D_x&=\D_x^1\oplus\cdots\oplus \D_x^r,\\
	\label{deRh1}\hol_x(\nabla^{0})&=\h_1\oplus\cdots\oplus \h_r,\\
	\label{deRh2}\hol_x(\nabla)&=\h'_1\oplus\cdots\oplus\h'_r\oplus\mathfrak{t}_x^\bot,\end{align} where, for each $i\in\{1,\dots,r\}$, 
	$$\h_i=\h'_i\oplus \Real J_x^i \subseteq\u(\D^i_x),\quad J_x^i={J_x}|_{\D^i_x}$$ is  an irreducible Riemannian holonomy algebra, and $\mathfrak{t}^\bot_x$ is the orthogonal complement to a one-dimensional subalgebra
	$$\t_x=\Real (a_1 J_x^1+\cdots +a_r J_x^r), \quad a_1,\dots,a_r\in\Real\backslash\{0\},$$
	in the commutative Lie algebra $\mathrm{span}\{J_x^1,\dots,J_x^r\}$.
\end{theorem}

\section{Pseudo-Hermitian structures}\label{secPsHerm} 

A prominent class of contact sub-Riemannian manifolds is provided by pseudo-Hermitian structures, which
emerge naturally in  the study of CR geometry. We explain this background briefly and 
relate common notions to our setting of extended Schouten connections. For detailed explanations on CR and 
pseudo-Hermitian geometry we refer to \cite{Tanaka,DT2006}.

Let $M$ be a smooth manifold of odd dimension $n=2m+1\geq 3$. A CR structure on $M$ can be defined as a pair $(\D,J)$ consisting 
of a contact distribution $\D$ and a complex structure $J:\D\to \D$, subject to the condition $$[X,JY]+[JX,Y]\in \Gamma(\D)$$ for any 
sections $X,Y$ of $\D$, and the vanishing of the Nijenhuis tensor
$$\mathcal{N}(X,Y)=[X,Y]-[JX,JY]+J([X,JY]+[JX,Y])=0$$ on $\D$. 
The first condition implies 
that any contact form $\theta$ for $\D$ gives rise
to a metric \begin{equation}\label{metrgCR} g=\d\theta(\cdot,J\cdot)\end{equation} on $\D$. For example, the smooth boundary
$M=\partial U$ of any strictly pseudoconvex domain $U$ in complex space $\mathbb{C}^{m+1}$ 
inherits naturally  a CR structure, and any defining function for the boundary $M=\partial U$ gives 
rise to a contact form $\theta$ such that $g$ given by \eqref{metrgCR} is a metric of Riemannian signature on $\D$. 

In the following, we assume $M$ to be equipped with a  strictly pseudoconvex CR structure $(\D,J)$ and some adapted contact form $\theta$
such that $g$ is positive definite. In this case  we call $(M,\D,J,\theta)$ an (abstract) CR manifold with pseudo-Hermitian structure, 
or simply a  pseudo-Hermitian manifold. 
It was discovered by Tanaka \cite{Tanaka}  and Webster \cite{Webster} that such a triple $(M,\D,J)$  admits for any choice of $\theta$ a
uniquely determined  linear connection. Independent of the construction procedures, this gives rise to  the so-called 
Tanaka--Webster connection $\nabla^{TW}$, which is known to be characterized by the following two conditions: 
$\nabla^{TW}$ satisfies the Koszul formula  (\ref{SchoutenCon})
for all sections $X,Y,Z$ of the contact distribution $\D$, and  for the Reeb vector 
$\xi$ of $\theta$ and any section $X$ in $\D$, we have
\[
\nabla_\xi^{TW}X=\frac{1}{2}\left( [\xi,X]-J[\xi,JX]\right)\ .
\]

These conditions define $\nabla^{TW}$ uniquely, no matter whether 
the Nijenhuis tensor $\mathcal{N}$ is assumed to vanish or not. 
In fact, we have 
\[g((\nabla^{TW}_XJ)(Y),Z)=-\frac{1}{2}\d\theta(X,\mathcal{N}(Y,Z) )\]
for $X,Y,Z\in\Gamma(\D)$ (cf. \cite{Tanno}), which shows that 
$J$ is $\nabla$-parallel if and only if $\mathcal{N}$ 
vanishes. 
In this case $\d\theta$ is $\nabla$-parallel as well.

By construction,  $\nabla^{TW}$ coincides with the Schouten connection for derivatives in $\D$-direction.
Hence,    $\nabla^{TW}$ is an extension, and  the endomorphism $N$ is given by 
$$NX=-\frac{1}{2}\left([\xi,X]+J[\xi,JX]\right),$$ 
which is usally called Webster torsion, denoted by $\bt$.

\begin{prop}
        The Tanaka--Webster connection $\nabla^{TW}$ coincides with the adapted connection~$\nabla^{\bt}$.
\end{prop}

\begin{proof}We compute 
\begin{eqnarray*}
(\mathcal{L}_\xi g)(X,Y) &=& \xi g(Y,X)-g([\xi,Y],X)-g(Y,[\xi,X])\\
&=& \xi \d\theta(Y,JX)-d\theta([\xi,Y],JX)-d\theta(Y,J[\xi,X])\\
&=& (\mathcal{L}_\xi \d\theta)(Y,JX)+d\theta(Y,[\xi,JX])-d\theta(Y, J[\xi,X])\\
&=& g(JY,[\xi,JX])-g(Y,[\xi,X])\\
&=& 2g(\bt X,Y)
\end{eqnarray*}
for all $X,Y\in \D$. Hence, the connection $\nabla^{TW}$ and $N=\bt$ satisfy (\ref{defA}). \end{proof}

The Ricci tensor of the curvature tensor $R^\bt$ of the Tanaka--Webster connection $\nabla^\bt$ satisfies, for $X,Y\in \D$, the identity 
\begin{equation}\label{RicTW}\Ric^{\bt}(X,Y)=\rho(X,JY)+(m-1)\bt(X,JY),\end{equation}
where $\rho$ is the pseudo-Hermitian Ricci-form given by
\begin{equation}\label{psHRicci}
\rho(X,Y)\ =\ \frac{1}{2}\tr_gg(R^\bt(J\cdot, \cdot)X,Y)\ .
\end{equation}
Note that $$\rho(X,Y)=\frac{1}{2}\tr_gg(R^\bt(X,Y,J\cdot,\cdot),$$ by the first Bianchi identity.
The trace of $\Ric^{\bt}$ on $\D$ is the Webster scalar curvature $scal^{\bt}$. 

In case that $$\rho=h \d\theta$$ for some  function $h$ on $M$, 
the pseudo-Hermitian form $\theta$ is called pseudo-Einstein structure on $M$. It follows $$h=\frac{1}{2m}scal^{\bt},$$ which 
needs not to be constant. 
In fact, we have $$R^{\bt}(\xi,\cdot,Y,X)=(\nabla^{\bt}_Y\bt)(X)- (\nabla^{\bt}_X\bt)(Y)$$ for $X,Y\in \D$, and $\Ric^{\bt}(\xi)=\tr_g\nabla^{\bt}\bt$.
The second Bianchi identity implies, for pseudo-Einstein spaces, $$\Ric^{\bt}(\xi,JX)=\frac{1}{2m}X(scal^{\bt})$$ for $X\in \D$.

In \cite{Leitner18} it was pointed out that a vector field $Z$ tangent to $\D$, which is $\nabla^{\bt}$-parallel in all $\D$-directions, needs not to be
covariantly constant in $\xi$-direction. In fact, for such $Z$, we have $$\nabla^{\bt}_\xi Z=-\frac{1}{m}\rho(Z).$$ Hence, if the Tanaka--Webster
connection is adjusted in $\xi$-direction by $$\hat{\nabla}_\xi Z:=\nabla^{TW}_\xi Z+\frac{1}{m}\rho(Z),$$ then $Z$ is a parallel vector field 
(in all directions) on $M$. Obviously, this adjustment $\hat{\nabla}$ is 
the extension of the Schouten connection with $$N=\bt+\frac{1}{m}\rho.$$ The connection  $\hat{\nabla}$ was called  basic in~\cite{Leitner18}.

    \begin{prop}
        The basic connection $\hat{\nabla}$ coincides with the Wagner connection $\nabla^W$.
    \end{prop}

\begin{proof}  We compute the endomorphism $N^W$, which determines the Wagner connection.
Note that $\d\theta(\d\theta^{-1})=-4m$. And, for the bivector $J=e_1\wedge Je_1+\cdots+e_m\wedge Je_m$, 
where $(e_1,\ldots,e_m)$ is some unitary basis of $(\D,g,J)$, we have $\d\theta(J)=2m$. Then, with (\ref{preWagner}), (\ref{NWagner}) and (\ref{psHRicci})
we obtain
\begin{eqnarray*}
N^W &=&\tfrac{1}{4m}R(\d\theta^{-1})\ =\ \bt+\tfrac{1}{4m}R^{TW}(\d\theta^{-1})\\ 
&=& \bt -\tfrac{1}{2m}R^{TW}(J)\ =\ \bt+\tfrac{1}{m}\rho\ ,
\end{eqnarray*}
the same expression as for the basic connection. \end{proof}

\section{Sub-Riemannian symmetric spaces}\label{secpsHermHolClass}
In this section we will review  results from 
\cite{BFG,FG,FGR} about sub-Riemannian symmetric spaces and their holonomy. We will provide more details about their classification in  Section~\ref{subsymholclass-sec}.
\subsection{Sub-symmetric spaces and quadruples}
\label{subsym-sec1}
An {\em isometry} between  sub-Riemannian contact manifold $(M,\theta,g)$ and $(\hat M,\hat \theta,\hat g)$  is a diffeomorphism $\phi:M\to \hat M$ such that 
\[\phi^*\hat \theta= \theta,\quad \phi^*\hat g=g.\]
This condition implies for the Reeb-vector field and for the sub-torsion that  
\[\phi_*\xi=\hat\xi, \qquad  \phi^*\hat \bt=\bt.\]
The isometries of $(M,\theta,g)$ are denoted by $\Iso (M,\theta,g)$, and if this group acts transitively on $M$ we call $(M,\theta,g)$ {\em homogeneous}. Given a point $o\in M$, a {\em sub-symmetry (at $o$)} is an isometry $\psi$ such that 
\[\psi(o)=o,\qquad\d\psi|_{\D_o}=-\mathrm{Id}.\]
A {\em sub-Riemannian symmetric space} is a homogeneous  sub-Riemannian contact manifold such that   every point in $M$ admits a  {\em sub-symmetry}.  A {\em locally sub-Riemannian symmetric space} is a  sub-Riemannian contact manifold $(M,\theta,g)$ such that every point admits a {\em local sub-symmetry} $\psi:U\to \psi(U)$ (homogeneity is not required). 

Now let $(M,\theta,g)$ be a sub-Riemannian contact manifold with sub-torsion $\bt$. Let $\nabla^\bt$ be the adapted affine connection, i.e.,~the unique affine connection that satisfies $\nabla^\bt g=0$,  $\nabla^\bt\xi=0$ and with torsion 
\[T(X,Y)= \d\theta(X,Y)\xi,\quad T(\xi,X)=\bt(X).\]
Let $R^\bt$ be the curvature of $\nabla^\bt$.
In analogy to the Riemannian situation in \cite[Theorem~2.1]{FG} it is shown that $(M,\theta,g)$ is locally sub-symmetric if and only if
\[\nabla^\bt_X T=0,\quad \nabla^\bt_X R^\bt=0,\quad\text{ for all $X\in \D$.}\]

The other important result in  \cite{FG}, which is also in analogy with the Riemannian case, is the purely algebraic description of sub-Riemannian symmetric spaces. This requires the following setting, which in 
\cite[Section 3]{FG} is called an {\em sub-orthogonal involutive Lie algebra} (or  {\em sub-OIL}, for short). We prefer the name sub-Riemannian symmetric quadrupel, in analogy to  symmetric triples of Riemannian symmetric spaces.
A {\em sub-symmetric quadrupel} $(\g,s, \k, B)$ consists of 
\begin{enumerate}
\item a real Lie algebra $\g$ and  an involution $s$ of $\g$, 
\item so that $\g=\h\+\p$, where $\h=\ker(s-\mathrm{Id})$ and $\p=\ker(s+\mathrm{Id})$, is the decomposition into the eigenspaces of $s$,
\item a subalgebra $\k\subseteq \h$   of co-dimension one that does not contain a non trivial ideal of $\g$,  
\item and  an $\mathrm{ad}_\k$-invariant inner product $B$  on $\p$,
\item such that bilinear form $\Theta:\p\times\p \ni (X,Y)\mapsto [X,Y]\mod \k \in \h/\k$ is nontrivial. 
\end{enumerate}
A sub-symmetric quadruple, for which $B$ is not only invariant under $\k$ but also under $\h$, is called {\em sub-torsion-free}. 
 For a sub-symmetric quadruple we have that $\h$ is a subalgebra and that
 \[ [\h,\p]\subseteq \p,\quad [\p,\p]\subseteq \h, \quad 
  \h=\k+[\p,\p].\]
 Since $\k\subseteq\so(\p,B)$ and $B$ positive definite, $\k$ is a compact Lie algebra, so that there is a $\k$-invariant complement $\mathfrak{m}$ to $\k$, i.e.,~$\g=\k\+\mathfrak{m}$. Moreover, this implies that 
that $\h=\k\+\langle Z\rangle $, where $Z$ centralizes $\k$. As a consequence $\Theta$ is $\ad_{\h}$-invariant.

For two sub-symmetric quadruples $Q=(\g,s, \k, B)$ and $\hat Q=(\hat\g,\hat s, \hat \k, \hat B)$ we say that 
that {\em $ Q$ restricts to $\hat Q$}
 if $\hat \g$ and $\hat \k$  are subalgebras of $\g$ and $ \k$, $\hat s$ is the restriction of $ s$, and $ \p=\hat \p$ and $B=\hat B$. 
 We call a sub-symmetric quadruple a {\em sub-symmetric transvection quadruple} if $[\p,\p]=\h$.
 Given the sub-symmetric quadruple $(\g,s, \k, B)$ one can can always restrict it to a transvection quadruple by setting  $\hat\h=[\p,\p]$ and $\hat \k=\k\cap \hat \h$ and the restricted quadruple $(\hat\g=\hat \h \+ \p,\hat k, s,B)$.

To every simply connected sub-Riemannian symmetric space $(M,\theta,g)$ one can  {\em associate} a sub-symmetric quadruple in the following way, \cite[Proposition~3.1]{FG}. Let
$G=\Iso(M,\theta,g)$ and $K$ the stabilizer group in $G$ of a point $o\in M$, and let $\g$ and $\k$ their corresponding Lie algebras. Let $\psi\in K$ be the sub-symmetry at $o$ and $s=\mathrm{Ad}_\psi$ the involution of $\g$ such that $\g=\h\+\p$ is the decomposition into $\pm 1$-eigenspaces of $s$. Then the projection
\[\pi:G\ni \psi\longmapsto \psi(o)\in M,\]
yields an isomorphism $\d\pi|_o$ between $\p$ and $ \D|_o$, and between $\mathfrak{m}$ and $T_oM$, that  pulls back the sub-Riemannian metric $g|_o$ to a $\mathrm{Ad}_K$-invariant bilinear form $B$ on $\p$. The $G$-invariant Reeb vector field induces a $\xi^*\in \mathfrak{m}$ such that $\h=\k\+\langle\xi^*\rangle$. Also the sub-torsion $\bt$ induces an $\Ad_K$-invariant symmetric bilinear form on $\p$, which we denote by $\bt^*$. Clearly we have that $\ad_\k\subseteq \so(\p,B)$, but $\ad_\xi$ is skew with respect to $B$ only  if the sub-torsion vanishes,   as we have, \cite[Lemma 4.5]{FG}.
\begin{equation}
B(\ad_\xi(X),Y)+B(X\ad_\xi (Y))=-2 \bt^*(X,Y).
\end{equation}
Hence $\ad_\h\subseteq \langle \bt^*\rangle \+\so(\p,B)$ and $\bt^*$ as well as the skew part of $\ad_\xi$, 
which is equal to $A_\xi:=\ad_\xi+\tau^*$,   commutes with $\k$.

Conversely, to every sub-symmetric quadruple $(\g,s, \k, B)$ one can define  simply connect sub-Riemannian symmetric space $(M=G/K,\theta,g)$, where $G$ is the unique simply connected Lie group with Lie algebra $\g$ and $K$ the unique connected subgroup of $G$ that has Lie algebra $\k$, such that   $(\g,s, \k, B)$ extends to the sub-symmetric quadruple associated to $(M,\theta,g)$, see  \cite[Proposition~3.2]{FG} for details. Moreover, two sub-symmetric spaces are isometric if one of their associated sub-symmetric quadruple restrict to the other. Hence,   the quadruple $(\g,s, \k, B)$ one can consider $\hat\h=[\p,\p]$ and $\hat \k=\k\cap \hat \h$ and the restricted quadruple $(\hat\g=\hat \h \+ \p,\hat k, s,B)$, so that one can without loss of generality assume that $[\p,\p]=\h$. 

\subsection{Curvature and holonomy}
In \cite{FGR} the  holonomy is defined as the holonomy of the  adapted affine connection, with $\xi$ parallel, which is hence the same as the holonomy of the adapted extended connection $\Hol(\nabla^\bt)$, $\hol(\nabla^\bt)$,  etc. The horizontal holonomy is obtained when restricting to horizontal loops, and  is the holonomy of the Schouten connection $\Hol(\nabla)$, $\hol(\nabla)$, etc. Let $(\g,s,\k,B)$ be a sub-symmetric quadruple associated to a sub-symmetric space $(M=G/H,\theta,g)$ and such that $\h=[\p,\p]$. 
The curvature of the adapted connection then is given as,  \cite[Proposition~4.1]{FGR}, 
\begin{eqnarray*}
R^\bt(\xi,X')|_o&=&0,
\\
R^\bt (X',Y')Z'|_o&=&-\left(\left[ \left[ X,Y\right],Z\right]\right)^*|_o -\theta^*([X,Y])|_o(\bt(Z))^*|_o,
\end{eqnarray*}
where $X,Y,Z$ are left-invariant vector fields on $G$, the vectors $X',Y',Z'$ are their images under $\d\pi|_o$ for $\pi:G\to G/H$ the canonical projection, $(.)^*$ denotes the map from $\g$ to fundamental vector fields on $M$, and $\theta^*=\pi^*\theta$.
Hence, 
\[
R^\bt (X',Y')Z'|_o=-\theta^*([X,Y])|_o\left( (\ad_\xi (Z))^* + \bt(Z')\right)|_o\mod (\ad_\k)^*|_o.
\]
Using this  and the Ambrose--Singer holonomy theorem 
 in \cite[Theorem~5.1]{FGR} it is shown that when $m\ge 2$ it is 
\begin{equation}
\label{symhol}
\hol(\nabla)=\ad_\k,\quad \hol(\nabla^\bt)=\ad_\k\+ \langle A_\xi\rangle,\end{equation}
 where 
$A_\xi=\ad_\xi+\tau$ is the skew part of $\ad_\xi$. Note that for sub-symmetric spaces which are sub-torsion-less we have $\hol(\nabla^\bt)=\ad_\h$.   Using this, in \cite[Theorem~6.1]{FGR} a classification of sub-symmetric spaces with irreducible holonomy is given, extending the earlier classification in \cite[Theorem 5.1]{FG} of those where $\h$ acts irreducibly on $\p$.  We will not give this classification here, as it  will follow from our results in Section~\ref{subsymholclass-sec}, where we do not assume irreducibility.

\section{Locally sub-symmetric spaces and locally homogeneous structures}\label{secLocsubsym}

In this section we consider locally sub-symmetric spaces.
Let $(M,\theta, g)$ be a contact 
  sub-Riemannian space with sub-torsion $\bt$, Reeb vector field $\xi$, and adapted connection $\nbt$ on $\D$ with curvature $\Rbt$ which is a section of $\Lambda^2T^*M\otimes \so(\D)$. In \cite{FG} the affine connection $\nbbt$ is considered that extends the adapted connection to sections of $TM$ by setting $\nbbt\xi=0$. It satisfies $\nbbt g=0$ and its torsion tensor is given by
  \[T(X,Y)=\d\theta (X,Y)\xi,\qquad T(\xi ,X)=\bt(X),\]
  for $X,Y\in \D$.
 Its curvature tensor $\Rbbt$ satisfies $\Rbbt(X,Y)\xi=\Rbbt(\xi,X)\xi =0$ and otherwise is equal to $\Rbt$.
 Then, as in \cite{FG}, $(M,\theta, g)$ is {\em locally sub-symmetric}  if and only if
 \begin{equation}
 \label{FGlocsym}\nbbt_X T=0,\qquad \nbbt_X  \Rbbt =0,\quad\text{ for all }X\in \D.\end{equation}
Using the   Schouten partial connection $\nabla$  on all tensor bundles of $\D$,  these equations are equivalent to
\begin{equation} 
\label{locsympart}
\nabla \bt=0,\qquad \nabla \d\theta=0,\qquad \nabla  R =0,
\end{equation}
where $R$ is the curvature tensor of the Schouten partial connection. For any extended connection $\nabla^N$, the connection induced on sections of tensor bundles of $\D$, is given by $\nabla^N=\L_\xi+ N\cdot$, where the dot denotes the natural action of an endomorphism on tensors. 
Equations~(\ref{locsympart}) imply for any extended connection $\nabla^N$  that
\begin{equation} 
\label{locsymN}
\nabla^N_X \bt=0,\qquad \nabla^N_X \d\theta=0,\qquad \nabla^N_X R =0,\quad \text{for all }X\in\D,
\end{equation}
and in particular for the adapted connection that 
\begin{equation}
\label{locsymRbt}
\nbt_X R^\bt(Y,Z)=0,\quad \text{for all }X,Y,Z\in\D
\end{equation}
It also applies to the Wagner connection $\nabla^W$ with $\nabla^W_\xi X=\nabla^\bt_\xi X+C(X)$, where \[C=\tfrac{1}{4m} R^\bt(\d\theta^{-1})\in \Gamma(\so(\D)),\] and which satisfies~(\ref{Wagner2}). Since $\nbt_X\d\theta=0$, we also have that $\nbt_X\d\theta^{-1}=0$, and together with~(\ref{locsymRbt}) this yields  $\nbt_XC =0$, and consequently
\begin{equation}
\label{nabWC}
\nabla^W_XC=0,\quad \text{for all }X\in\D.
\end{equation}
Hence, using~(\ref{RNXY}), for the Wagner connection of a locally symmetric sub-Riemannian space we have 
\begin{equation} 
\label{locsymW}
\nabla^W_X \bt=0,\qquad  \nabla^W_X C=0,\qquad \nabla^W_X \d\theta=0,\qquad \nabla^W_X R^W =0,\quad \text{for all }X\in\D.
\end{equation} 
Setting $N^W:=\bt +C$, differentiating these equations and skewing yields
   for any 
 $X,Y\in \D$ 
 that 
 \[
 [ R^W(X,Y) ,N^W]=\d\theta(X,Y) \nabla^W_\xi N^W,\qquad R^W(X,Y)\cdot \d \theta =  \d\theta(X,Y) \nabla^W_\xi \d\theta \]
 and \[(R^W(X,Y)\cdot R^W) (U,V) = \d\theta(X,Y)(\nabla^W_\xi R^W)(U,V),\]
  where $U,V\in \D$ and the dot denotes the natural action of endomorphism on the corresponding tensors.
 When substituting $\d\theta^{-1}$ for $X,Y$ in these  equations, by the defining equation~(\ref{Wagner2}) for the Wagner connection,  the left-hand-side vanishes, whereas $\d\theta(\d\theta^{-1})=-4m$,  so that 
 \begin{equation}
 \label{allWparallel}
 \nabla^W_\xi N^W=0,\qquad \nabla^W_\xi \d\theta=0,\qquad \nabla^W_\xi R^W(X,Y)=0,\end{equation}
and hence
 \begin{equation}
 \label{RWW}
 [ R^W(X,Y) ,N^W]=0,\qquad R^W(X,Y)\cdot \d \theta = 0,\qquad 
 (R^W(X,Y)\cdot R^W)(U,V)= 0,
 \end{equation}
for all $X,Y,U,V\in \D$. 
 It also implies that
 \begin{equation}
 \label{Liexi}
 \L_\xi N^W=0,\qquad  \L_\xi \d\theta +N^W\cdot \d\theta=0,\qquad \L_\xi R^W+N^W\cdot R^W=0.\end{equation}
Using Cartan's formula, we get 
$\L_\xi\d\theta = \iota_\xi\d^2\theta +\d  \iota_\xi\d\theta=0$, so that
\begin{equation}
\label{Wdtheta}
N^W\cdot \d\theta =0.
\end{equation}
Moreover, we consider the affine connection $\nbW$ defined by the Wagner connection and the condition $\nbW\xi=0$. Its torsion is given by
\[T^W(X,Y)=\d\theta(X,Y)\xi,\qquad T^W(\xi,X)=N^W(X).\]
From the above, we have that the torsion tensor $T^W$ is parallel for $\bnab$. 
By Lemma~\ref{lemRAxi} and equation~(\ref{RNxiX}) for the adapted connection, we get that $\L_\xi\nabla=0$ and hence, from the same equation and~(\ref{locsymW}) we get that $R^W(\xi,X)=0$.   This implies that also the curvature of $\bnab$ is  $\bnab$-parallel. Moreover,
the second Bianchi identity for $\nbW$ then gives
\[(\nbW_\xi R^W)(X,Y)=R^W(N^WX,Y)+R^W(X,N^WY),\]
which vanishes by the third equation in~(\ref{allWparallel}), so that
\begin{equation}\label{RWR}
R^W(N^WX,Y)+R^W(X,N^WY)=0\qquad N^W\cdot R^W (X,Y) =0,
\end{equation}
where the second equation holds because of the first and the first in~(\ref{RWW}).
With the help of these equations one can show that a locally symmetric sub-Riemannian space defines a a sub-symmetric quadruple. The process follows the construction of a Lie triple system from a locally symmetric Riemannian space. 
\begin{theorem}
Any locally symmetric, contact sub-Riemannian space $(M,\theta, g)$ defines a sub-symmetric quadruple $(\g,s,\k,g_x)$ and hence a sub-Riemannian symmetric space $G/K$. 

Moreover, $\k$ is the horizontal holonomy algebra and $(M,\theta, g)$ is locally isometric to $G/K$.
\end{theorem}
\begin{proof}
We fix a point $x\in M$  and define the vector spaces
\[ \p:=\D_x,\quad \k:=\mathrm{span}\{R^W(X,Y)\mid X,Y \in \D_x\},\quad \h:=\k\+\langle \xi_x\rangle,\]
so that $T_xM=\p\+\langle \xi_x\rangle$ and $\k\subseteq \so(\p)$.
We set  $\Theta (X,Y):=\d\theta_x(X,Y)$
and
define a  bracket on $\g:=\p\+\h$ as follows, where we omit the point $x$ at the tensors and abuse notation by using the same bracket symbol as above for the commutator of vector fields and  endomorphisms. For $W=\bt+C$ and $R^W$ the curvature tensor of the Wagner connection, $X,Y\in \p$ and $A,B\in \k$, we define
\begin{equation}\label{Liebracket}
\begin{array}{rcl}
\left[X,Y\right]&=& -R^W(X,Y)+\Theta (X,Y)\xi,
\\
\left[\xi,X\right]&=&N^W(X)\ =\ -\left[X,\xi\right],
\\
\left[\xi,A\right]&=&0\ =\ [A,\xi ], 
\\
\left[A,X\right]&=& AX\ =\ - [X,A],
\\
\left[A,B\right]&=& AB-BA.
\end{array}
\end{equation}
We need to check that this satisfies the Jacobi identity. 
Since   $[[X,Y],Z]=-R(X,Y)Z)$,  the Jacobi identity for $X,Y,Z\in \p$ is satisfied  because of~(\ref{RNXY}). If two of the arguments are from $\k$, the Jacobi identity holds  trivially. 
For the remaining cases we have for $X,Y\in \p$ that
\begin{eqnarray*}
[[X,Y],\xi]+[[Y,\xi],X]+[[\xi,X],Y]
&=&
[X,N^WY]+[Y,N^WX]
\\
&=&-R^W(N^WX,Y)-R^W(X,N^WY) +(N^W\cdot \d\theta)(X,Y)\xi
\\
&=&
0,
\end{eqnarray*}
because of~(\ref{Wdtheta}) and~(\ref{RWR}).
For 
$A\in \k$ it is
\[
[[X,Y],A]+[[Y,A],X]+[[A,X],Y]
=
A\cdot R^W(X,Y)-A\cdot \Theta (X,Y)
=
0,
\]
and 
\[
[[X,\xi],A]+[[\xi,A],X]+[[A,X],\xi]
=
[A,N^W]
=
0
\]
by~(\ref{RWW}).
Hence, $\g$ is a Lie algebra with an involution $s$ defined by $s|_\h=\id$ and $s|_\p=-\id$, a co-dimension one subalgebra $\k\subseteq \h$, with $g_x$ as a $\ad_\k$ invariant inner product, i.e.,~satisfying the conditions for a sub-symmetric quadruple $(\g,s,\k,g_x)$. 

Moreover, we have that $\nabla^WR^W=0$ so that $\k$ is the holonomy algebra of $\nabla^W$ and therefore, by~Theorem~\ref{thholW} equal to the horizontal holonomy, i.e.,~to the holonomy algebra of the Schouten partial connection.

By taking $G$ to be the unique simply connected group with Lie algebra $\g$ and $K$ to be the unique connected subgroup corresponding to the subalgebra $\k$, we obtain a sub-symmetric space $G/K$. 

By construction we have a linear isometry between $T_xM$ and $T_{[e]}(G/H)$, and we need to show, similarly to the Riemannian situation, that it extends to a local isometry.
Above in~(\ref{allWparallel})  we have in seen that  the affine connection $\nbW$ defined by the Wagner connection and the condition $\nbW\xi=0$ has parallel torsion and curvature. Moreover, the metric defined as
\[\overline{g}=\theta^2+ g\]
is parallel for $\bnab$. Hence, $\bnab$ is an Ambrose--Singer connection for the Riemannian manifold $(M,\overline{g})$, which, as a consequence, is a locally homogeneous Riemannian manifold, \cite{AmbroseSinger58}, see also 
\cite{tricerri-vanhecke83}. As such it is locally isometric to a Riemannian homogeneous space that is given by $G/K$, since its construction follows~(\ref{Liebracket}). 
Hence, the locally sub-symmetric space is locally isometric to the globally sub-Riemannian space $G/K$.
\end{proof}
In the proof we have used the following observation, which follows from~(\ref{locsymW}) and~(\ref{allWparallel}), and which is worth to be stated separately.
\begin{cor}
A contact sub-Riemannian space $(M,\theta, g)$ is locally symmetric if and only if 
the Wagner connection defines  an Ambrose--Singer connection (i.e., a homogeneous structure)  for the Riemannian metric $\theta^2+g$ on $M$ and with $\xi$ is parallel. 
\end{cor}


\section{Classification of pseudo-Hermitian holonomy} 
\label{SecClass} 
In this section we will provide a classification of the holonomy algebras of the Schouten and the Tanaka--Webster connections of a pseudo-Hermitian manifold. The K-contact case, i.e.,~when $\bt=0$, is fully described in  Section \ref{secbt0} above, so in Theorem~\ref{ThClass} below we will focus on the case when $\bt\not=0$.

Let $(M,J,\theta)$ be a pseudo-Hermitian manifold of dimension $2m+1\geq 7$. 
The sub-Riemannian metric is $g(X,Y)=\d\theta(X,JY)$.
 We denote by $\u(\D)$ the 
vector bundle with fibres being the centralizer $\u(\D_x)\cong\u(m)$ of $J_x$ in  $\so(\D_x)$.
Let $\nabla^\bt=\nabla^{TW}$ be the Tanaka--Webster connection. Recall that $\nabla^\bt J=\nabla^\bt\d\theta=0$, so that $R^\bt$ is a section of $\Lambda^2\D^*\otimes \u(\D)$ and $\hol(\nabla^\bt)\subseteq\u(m)$.

Let $E_x\subseteq \D_x$ be a Lagrangian subspace, i.e.,~a maximal vector subspace such that the restriction of $\d\theta$ to $E_x$ is zero, 
\[ \d\theta_x|_{E_x\wedge E_x}=0.\]
We obtain the  $g$-orthogonal decomposition
$$\D_x=E_x\oplus JE_x.$$
The maximal subalgebra of  
\[\u (\D_x)= \left\{ \begin{pmatrix} A&B \\-B &A\end{pmatrix} \mid A\in \so(\D_x), B\in \odot^2 \D_x\right\},\] 
 preserving this decomposition is of the form 
$$\left\{\left(\begin{matrix}
A&0\\0&A
\end{matrix} \right)\mid A\in\so(E_x) \right\},$$
where the matrices are written with respect to an orthonormal basis $e_1,\dots,e_m,Je_1,\dots,Je_m$ of $\D_x$ with $e_1,\dots,e_m\in E_x$. 
This defines the embedding $\so(E_x)\hookrightarrow\su (\D_x)$, and we will denote by $\so(E_x)$ the image of this embedding. If the subspace $E_x$ is not specified, then we denote such subalgebra by $\so(m)$. We denote the 1-dimensional subalgebra $\Real J\subseteq\u(m)$ by $\u(1)$.

\begin{theorem} \label{ThClass} Let $(M,J,\theta)$ be a pseudo-Hermitian manifold of dimension $2m+1\geq 7$. Suppose that $\bt\neq 0$. Then there are the following possibilities for the  holonomy algebras $\hol(\nabla^\bt)$ of the Tanaka--Webster connection and $\hol(\nabla)$ of the Schouten connection: 
	\begin{itemize}
		\item[ 1.] $\hol(\nabla^\bt)=\u(m)$, and either $\hol(\nabla)=\u(m)$ or $\hol(\nabla)=\su(m)$;
		\item[ 2.] $\hol(\nabla^\bt)=\hol(\nabla)=\su(m)$;
		\item[ 3.] $\hol(\nabla^\bt)=\so(m)\oplus\u(1)$ and $\hol(\nabla)=\so(m)$; 
		\item[ 4.] $\hol(\nabla^\bt)=\hol(\nabla)=\so(m)$.
	\end{itemize}
	In particular, either $\hol(\nabla^\bt)=\hol(\nabla)$ or $\hol(\nabla^\bt)=\hol(\nabla)\+\Real J$.
\end{theorem}

\begin{proof}
For the proof we 
define the section $R_\bt$ of $\Lambda^2\D^*\otimes \so(\D^*)$ by
	$$R_\bt(X,Y)=-\tfrac12\left(JX \wedge \bt Y + \bt X \wedge JY+ X \wedge\bt J Y+ \bt J X \wedge Y\right),$$ 
	where we use the isomorphism $\D\wedge \D\cong \so(\D)$ given by
	\[(X\wedge Y) (Z):= g(X,Z)Y-g(Y,Z)X.\]
Because
\[\bt=-\tfrac{1}{2}\left([\xi,X]+J[\xi,JX]\right)\]
anti-commutes with $J$, the tensor $R_\bt$ is a section of $\Lambda^2\D^*\otimes \u(\D)$ and satisfies 
 \[R_\bt(JX,JY)=-R_{\bt}(X,Y).\]
%
%
%
Moreover, a direct computation yields
\[
\underset{X Y Z}{\Su} R_\bt(X,Y)Z=\underset{X Y Z}{\Su} \d\theta(X,Y)\bt Z, \\
\]
for all $X,Y,Z\in\D$.
Hence we define $$R_0=R^\bt-R_\bt,$$ and  recall that the curvature of the Tanaka--Webster connection satisfies the following Bianchi identity
\begin{equation}\label{BianchiTWnew1}
\underset{X Y Z}{\Su} R^\bt(X,Y)Z=\underset{X Y Z}{\Su} \d\theta(X,Y)\bt Z,  \quad\text{for all } X,Y,Z\in\D,
\end{equation}
so that
 \begin{equation}\label{R0Bianchi}
\underset{X Y Z}{\Su} R_0(X,Y)Z=0,  \quad\text{for all } X,Y,Z\in\D.\end{equation}
As a consequence, 
$R_0$ is a section of  $\R(\u (\D))$.
 In particular, 
\begin{equation}
\label{R0J}
R_0(JX,JY)=R_0(X,Y),\quad \text{for all } X,Y,Z\in\D.\end{equation}
We conclude that, for each point $y\in M$ and all vectors $X,Y\in \D_y$, it holds
\begin{eqnarray}
\label{R0hol}
2R_0(X,Y)&=&R^\bt(JX,JY)+R^\bt(X,Y)\in\hol_y(\nabla^\bt),\\
\label{Rthol}
2R_\bt(X,Y)&=&-R^\bt(JX,JY)+R^\bt(X,Y)\in\hol_y(\nabla^\bt),
\end{eqnarray}
i.e.,, both  $R_0$, and $R_\bt$ take values in $\hol(\nabla^\bt)$.

For $R\in \Lambda^2\D^*\otimes \so(\D)$ we define its Ricci tensor
\[\Ric(R)(X,Y)=\mathrm{tr}(R(X,\ .\ )Y).\]
Using that 
\[ g(R_\bt(X,Je_i)Y,Je_i) = -g(R_\bt(JX,e_i)JY,e_i),\]
it is easy to check that
\[
\Ric(R_\bt)(X,Y)=
(m-1)\bt(X,JY),\quad \text{for all } X,Y\in\D.\]
From this and \eqref{RicTW}  it follows that
\begin{equation}
\label{RicR0}\Ric(R_0)(X,Y)=\rho(X,JY),\quad\text{for all } X,Y\in\D.\end{equation}

We will now use $R_\bt$ to establish the four possibilities for the holonomy algebra of the Tanaka--Webster connection.
Let us fix a point $y\in M$ and look at the structure of the symmetric endomorphism $\bt_y$ of $\D_y$. Denote by $E_\lambda$ the eigenspace of $\bt_y$ corresponding to the eigenvalue $\lambda$. Since $\bt$ anti-commutes with $J$, eigenvalues come in pairs $\pm\lambda$ and 
$$JE_\lambda=E_{-\lambda}.$$
Hence we have an orthogonal decomposition 
\[\D_y=E^+\+E^-\+\ker(\bt),\]
where $E^\pm$ is the sum of the eigenspaces to positive (negative) eigenvalues, and consequently a $g$-orthogonal decomposition with a Lagrangian subspace $E_y$,
$$\D_y=E_y\oplus JE_y$$ such that 
$E^+\subseteq E_y$ and $E^-\subseteq JE_y$.
In particular, we can fix a basis $e_1,\dots,e_m$ of $E_y$ consisting of eigenvectors of $\bt_y$ corresponding to a non-negative eigenvalues $\lambda_1, \ldots, \lambda_m$, i.e.,
\begin{equation}\label{eigenbasis}
\tau e_j=\lambda_j e_j,\quad \tau Je_j=-\lambda_j Je_j,\quad  \lambda_j\geq 0.\end{equation}
Not that if  $\bt_y$ is injective, then the Lagrangian subspace  $E_y=E^+$ is defined uniquely and we denote it by $E^{\bt}_y$. We do not assume this, but it will be relevant later.

It is easy to check that 
\begin{align}\label{Rtaujk}R_\bt(e_j,e_k)&=\tfrac{\lambda_k-\lambda_j}{2}(e_j\wedge J e_{k}+e_{k}\wedge Je_{j}),\\
\label{Rtaujkm} R_\bt(e_j,Je_{k})&=\tfrac{\lambda_j+\lambda_k}{2}(e_j\wedge e_{k}+Je_{j}\wedge Je_{k}).
\end{align}
Since $J=\sum_{k=1}^m e_k\wedge Je_k$, this implies \begin{equation}\label{Rtau(J)}
R_\bt(J)=0,
\end{equation}
and
hence 
$R_\bt$ is a section of $\Lambda^2\D^*\otimes\su(\D)$.
With our assumption $\bt\neq 0$, there exists a point $y\in M$ and  $j$ such that $\lambda_j>0$. Since all $\lambda_k$ are non-negative, so that $\lambda_j+\lambda_k\not=0$, relation~(\ref{Rthol})  implies that for each $k$ we have,
 $$e_j\wedge e_{k}+Je_{j}\wedge Je_{k}\in \hol_y(\nabla^\bt).$$
Note that if $i,j,l$ are pairwise different, then $$[e_j\wedge e_k,e_j\wedge e_l]=e_k\wedge e_l,$$
which ensures that for all $k,l=1, \ldots, m$ we have
$$e_k\wedge e_{l}+Je_{k}\wedge Je_{l}\in \hol_y(\nabla^\bt).$$
We recall that the subalgebra $\so(m)\subseteq \su (\D_y)$ is spanned by these  endomorphisms
to conclude that there exists an orthogonal decomposition $\D_y=E_y\oplus JE_y$ such that
$$\so(E_y)\subseteq\hol_y(\nabla^\bt).$$ 
We can split $\u(m)$ into $\so(m)$-modules 
$$\u(m)=\so(m)\oplus (\odot^2 \Real^m)_0\oplus\Real J,$$
where $(\odot^2 \Real^m)_0$ denotes the symmetric trace-free $m\times m$ matrices.
Since $\so(m)\subseteq \hol_y(\nabla^\bt)$, the projection of $\hol_y(\nabla^\bt)$ to $ (\odot \Real^m)_0$ is invariant under $\so(m)$. However, 
if $m\geq 3$,  the $\so(m)$-module is $(\odot \Real^m)_0$ is irreducible, so that $\hol(\nabla^\bt)$ is one of the Lie algebras 
\begin{equation}\label{EqList}\u(m),\quad \su(m),\quad \so(m)\oplus\Real J,\quad \so(m).\end{equation}
This proves the statement in the theorem about the four possible holonomy algebras of the Tanaka--Webster connection.

For the remaining statements about  the holonomy $\hol_y(\nabla)$ of the Schouten connection $\nabla$, we will use the tensor $R_0$, but  first recall from~(\ref{RWtxy}) that
 \begin{equation}\label{RWRWTtau}
R^W(X,Y)=R^\bt(X,Y)+d\theta(X,Y)C,\quad \text{for all } X,Y\in\D,\end{equation}
where $R^W$ is the curvature of the Wagner connection and
$C\in\u(\D)$.
With~(\ref{R0hol}) and~(\ref{Rthol}) this implies that for all $X,Y\in \D_y$ it holds
\begin{align}\label{2R0+2theta}
2R_0(X,Y)+2\d \theta(X,Y) C&=R^{W}(JX,JY)+R^{W}(X,Y)\in\hol_y(\nabla),\\
\label{Rbthol}
2R_\bt(X,Y)&=-R^{W}(JX,JY)+R^{W}(X,Y)\in\hol_y(\nabla),
\end{align}
where the containment in $\hol_y(\nabla)$ is guaranteed by Theorem~\ref{thholW}.
From the last equality it follows that 
$R_\bt$ is a section of $\Lambda^2\D^*\otimes \hol(\nabla)$. On the one hand, this implies that
$$\so(m)\subseteq\hol(\nabla),$$ and consequently $\hol(\nabla)$ also belongs to the list \eqref{EqList}.
On the other hand, from the definition of the Wagner connection, it implies  for each $y\in M$, that
\begin{equation}\label{R_0(alpha)}
(R_0(\alpha))_y\in \hol_y(\nabla),\quad\textrm{if } \d\theta(\alpha)=0.
\end{equation}
Moreover, 
from (\ref{Rtau(J)}) and \eqref{RWRWTtau} it follows that
\begin{equation}\label{CR_0}
C=-\tfrac{1}{2m}R^\bt(J)=-\tfrac{1}{2m}R_0(J).  
\end{equation}

Since we have established that 
$\hol(\nabla^\bt)$ as well as $\hol(\nabla)$ are from the list~(\ref{EqList}), for the proof of the theorem it remains to exclude the case $\hol(\nabla)=\so(m)\+\u(1)$ and to show that $\hol(\nabla)=\so(m)$ implies that $\hol(\nabla^\bt)\subseteq \so(m)\+\u(1)$. 
We will need the following lemma.

\begin{lem}\label{LetEbt} 
Assume that $\hol(\nabla)\subseteq\so(m)\oplus\u(1)$. Then on the open set $U_\bt=\{\tau\not=0\}$  there is  a vector distribution $\E^\bt$ of  Lagrangian subspaces such that 
\begin{equation}\label{DecDEbt}
	\D|_{U_\bt}=\E^{\bt}\oplus J\E^{\bt},\end{equation}
	and 
	$$\bt =\lambda\left( \begin{matrix} \1& 0\\0&-\1\end{matrix}\right),$$
	for a positive $\lambda\in C^\infty(U_\bt)$. In particular, 
	$\bt^2=\lambda^2\id_\D$. 
 \end{lem}

\begin{proof} 
At each point $y\in U_\bt$ we can chose an eigen-basis for $\bt$ as in~(\ref{eigenbasis}). 
Equations~(\ref{Rtaujk}) and (\ref{Rtaujkm}) show that also $R_\bt\not=0$ on $U_\bt$. On the other hand,   from~(\ref{Rbthol}) and by the assumption we know that $R_\bt\in \Lambda^2\D^* \otimes (\so(m)\+\u(1))$, which  by equation~(\ref{Rtaujk}) means that  $\lambda_j=\lambda_k$, i.e., 
	$$\bt_y=\lambda\left( \begin{matrix} \1& 0\\0&-\1\end{matrix}\right),\quad\lambda > 0.$$
This works at any $y\in U_\bt$ and hence defines a positive function $\lambda\in C^\infty(U)$. The distribution $\E^\bt$ is given as kernel of $\bt-\lambda \id_\D$ and we have $J\E^\bt=\ker(\bt+\lambda\id_\D)$.
 \end{proof}
%
%
%
%
%
%
%
%
%
%

  The proof of the theorem will now follow from the next lemma.

\begin{lem}
\label{mainlemma} If $\hol(\nabla)\subseteq\so(m)\oplus\u(1)$, then $\hol(\nabla)=\so(m)$, 
	$\hol(\nabla^\bt)$ is one of $\so(m)$ or $\so(m)\oplus\u(1)$, and $(M,\theta,g)$ is locally sub-symmetric.
\end{lem}

\begin{proof}
Let $U_\bt=\{\bt\not=0\}$ as in  Lemma \ref{LetEbt}, which yields the $g$-orthogonal decomposition $\D|_{U_{\bt}}=\E^\bt\+J\E^\bt$ of~(\ref{DecDEbt}), and for the moment we will consider all tensor fields restricted to $U_\bt$.
We have that 
$R_0\in\R(\u(\D))$, so that 
$$R_0:\u(D)\to\u(\D)$$ is a symmetric endomorphism, and with respect to the decomposition
$$\u(\D)=\su(D)\oplus\langle J\rangle,$$
it holds 
\begin{equation}\label{R0Matr}
R_0=\left( \begin{matrix} S& F^T\\F&s\end{matrix}\right),\end{equation}
where $S$ a symmetric endomorphsm of $\su(\D)$, $F$  is a  section of $\su(\D)^*$, uniquely defined by the trace-free part of the Ricci tensor of $R_0$, and $s$ is proportional to the trace of the Ricci tensor of $R_0$.
From this and  equation \eqref{R_0(alpha)} 
 it follows that $S$ is a section of $\Lambda^2\D^*\otimes \so(E^\bt)$. 
Let $X,Y,Z\in E^\bt$ be linearly independent vectors (recall that $m\geq 3$).
The bi-vectors of the form
$$X\wedge Y+JX\wedge JY$$
belong to $\su (\D)$.
Consequently, by~(\ref{R0J}), 
\begin{eqnarray*}
2R_0(X\wedge Y)&= & R_0(X\wedge Y+JX\wedge JY)\\
&=&S(X\wedge Y+JX\wedge JY)+F(X\wedge Y+JX\wedge JY)
\\
&\in& \so(E^\bt)\oplus\langle J\rangle.
\end{eqnarray*}
Using this and projecting to $JE^\bt$  the Bianchi identity \eqref{R0Bianchi} written for the vectors $X,Y,Z$, we obtain 
$$\underset{X Y Z}{\Su}F(X\wedge Y+JX\wedge JY)Z=0.$$ By the linear independence of $X$, $Y$ and $Z$, this implies
$$F=0.$$
Hence, $R_0$ is a section of  $\Lambda^2\D^*\otimes \so(E^\bt)\oplus\langle J\rangle$. From this and \eqref{CR_0} it follows that $$C=fJ,$$ where $f\in C^\infty(U_\bt)$.

Let $L\subseteq \odot^2\D$ be the subbundle of $g$-symmetric endomorphisms of $\D$ that commute with $\so(\E^\bt)$ and anti-commute with $J$. 
The fibres of this bundle are invariant under $\so(m)\+\u(1)$ and hence under $\hol(\nabla)\subseteq \so(m)\+\u(1)$. Indeed, by definition $B\in L$ commutes with $\so(m)$ and $[J,B]=2JB$ clearly anti-commutes with $J$. Since $J$ is parallel and $\so(m)\subseteq \hol(\nabla)$, $J$ commutes with $\so(m)$, so that $JB$ also commutes with $\so(m)$. We conclude that $L$ is invariant under parallel transport.

Moreover, it is easy to check that in the decomposition \eqref{DecDEbt}, 
$$\left( \begin{matrix} a\1& b\1\\b\1 &-a\1\end{matrix}\right)=a\bt_y+bJ\bt_y.$$
Hence, over $U_\bt$, the sections $\bt$ and 
$J\bt$ provide a frame for $L$. Since $L$ was shown to be invariant under $\nabla$-parallel transport, this implies that over $U_\bt$ we have 
$$\nabla \bt=\alpha\otimes \bt+\beta\otimes J\bt,$$
where $\alpha,\beta\in \Gamma(\D^*|_{U_\bt})$. Denote by $V,U\in\Gamma(\D|_{U_\bt})$ the dual vector fields to $\alpha$ and $\beta$, respectively. 
Using Lemma \ref{lemRAxi}, we obtain
\[g(R^\bt(\xi,X)Y,Z)=\alpha(Y)g(\bt X,Z)+\beta(Y)g(J\bt X,Z)-\alpha(Z)g(\tau X,Y)-\beta(Z)g(J\tau X,Y)\]
for all $X,Y,Z\in\Gamma(\D|_{U_\bt})$.
This means that
$$R^\bt(\xi,X)=\alpha^\sharp\wedge \bt X+\beta^\sharp \wedge J\bt X,\quad\text{for all } X\in\Gamma(\D|_{U_\bt}),$$
where the sharp denotes the metric dual.
Since $R^\bt(\xi,X)\in\u(D)$, it holds 
$$[J,R^\bt(\xi,X)]=0,$$ which implies that $\beta=J^*\alpha$. Hence, together with $C=fJ$, from \eqref{RWRAC} it follows for the curvature of the Wagner connection that
\begin{equation}
\label{RWRWTtau1}
R^{W}(\xi,X)=\alpha^\sharp \wedge \bt X+(J\alpha^\sharp)\wedge J\bt X-X(f)J,\quad\text{for all } X\in\Gamma(\D|_{U_\bt}).\end{equation}
If $\alpha \neq 0$, then such elements generate either $\su(m)$ or $\u(m)$. This is in contradiction with the assumption $\hol(\nabla)\subseteq\so(m)\oplus\u(1)$, so that 
$\alpha=\beta=0$, and consequently,
$$\nabla\bt=0.$$ 
This implies that $\bt$ is non-vanishing, i.e.,~$U_\bt=M$,  that 
the distribution $E^\bt$ is $\nabla$-parallel and that the subbundle $\so(E^\bt)\subseteq\u(\D)$ is $\nabla$-parallel. That $E^\bt$ is $\nabla$-parallel implies that 
$$\hol(\nabla)=\so(m).$$  
Moreover recall~(\ref{Rtau(J)}), which now implies that  $R_\bt$ is a section of $\Lambda^2\D^*\otimes \so(E^\bt)$, which, since $J$ and $\bt$ are $\nabla$-parallel, 
implies that $R_\bt$ is also $\nabla$-parallel. 
Hence in order to show that $(M,J,\theta)$ is locally sub-symmetric, by the result reviewed in Section~\ref{subsym-sec1} or \cite[Theorem~2.1]{FG}, it remains to show that $\nabla R^\bt=\nabla R_0=0$.

By Theorem \ref{ThAS1},
$R^W(\xi,X)=-X(f)J$ takes values in $\hol(\nabla)=\so(m)$, so that by~(\ref{RWRWTtau}) we must have  $X(f)=0$ for all $X\in \Gamma(\D)$. Since $\D$ is contact, this implies that $f$ is constant. We have seen that $R_0\in
\R(\so(m)\oplus \Real J)$. 
 Now recall that the space $\R(\so(m)\oplus \Real J)$ is one-dimensional, see \cite[Table 1]{C-S} or \cite{Salamon89} for a proof, and  its elements $R$ are uniquely determined by their Ricci tensor. By~(\ref{RicR0}) the Ricci tensor  $\Ric(R_0)$ of $R_0$  is equal to the Ricci-form $\rho$, and hence, with~(\ref{psHRicci}), determined by
$R_0(J)$. 
By~(\ref{CR_0}) an since $f$ is constant, we have that $R_0(J)=4mfJ$ is $\nabla$-parallel, so that $\nabla R_0=0$.
We conclude that  $(M,\theta,g)$ is locally sub-symmetric. 
Since $\nabla\bt =0$, Theorem~\ref{thHolnabl}  implies that $\hol(\nabla^\bt)$ is one of $\so(m)$ or $\so(m)\oplus\u(1)$. This proves the lemma. \end{proof}
To finish the proof of the theorem, recall that we have seen that both $\hol(\nabla^\bt)$ and $\hol(\nabla)$ are on the list~(\ref{EqList}). 
Lemma~\ref{mainlemma} then excludes the case $\hol(\nabla)=\so(m)\+\u(1)$ and shows that $\hol(\nabla)=\so(m)$ only if  $\hol(\nabla^\bt)\subseteq \so(m)\+\u(1)$.\end{proof}

 \section{Consequences and example}\label{SecCons}

 \subsection{Consequences}
 From Theorem~\ref{ThClass} and Lemma~\ref{mainlemma} in its proof we obtain the following consequences.
 
 \begin{prop}\label{PropSym} Let $(M,J,\theta)$ be a pseudo-Hermitian manifold of dimension $2m+1\geq 7$ with $\bt\neq 0$. Then $\hol(\nabla)=\so(m)$ if and only if $(M,J,\theta)$ is locally sub-symmetric. \end{prop}

  \begin{proof} 
  If $\hol(\nabla)=\so(m)$, Lemma~\ref{mainlemma} applies and yields that $(M,J,\theta)$ is locally sub-symmetric.
  
For the converse, suppose that $(M,J,\theta)$ is locally sub-symmetric. By Theorem~\ref{ThClass} we have to exclude the cases that   $\hol(\nabla)$ is $\su(m)$ or $\u(m)$. In both cases  
  $\hol(\nabla)$ is irreducible and commutes with $\bt$. Since $\bt$  is symmetric,   Schur's lemma  implies that $\bt$ is a real multiple of the identity. Since $\bt$ anti-commutes 	with $J$, this shows that $\lambda=0$, which contradicts the assumption that $\bt\not=0$. Thus, $\hol(\nabla)=\so(m)$.  
 \end{proof}

 \begin{prop} Let $(M,J,\theta)$ be a pseudo-Hermitian manifold of dimension $2m+1\geq 7$ with $\bt\neq 0$. If $\nabla\bt=0$, then $(M,J,\theta)$ is locally sub-symmetric.
 \end{prop}
 
 \begin{proof} Suppose that $\hol(\nabla)$ is irreducible. Since $\bt$ is a  parallel symmetric field of  endomorphisms, it commutes with $\hol(\nabla)$, and consequently it is proportional to $\id_{\Gamma(\D)}$. Since $\bt$ anticommutes with $J$, this implies $\bt=0$, which is a contradiction. Hence, $\hol(\nabla)$ is not irreducible. From Theorem~\ref{ThClass} it follows that $\hol(\nabla)=\so(m)$.  Proposition \ref{PropSym} then yields that  $(M,J,\theta)$ is locally sub-symmetric. \end{proof}

 \begin{prop}\label{PropClassCases} 
 	Let $(M,J,\theta)$ be a pseudo-Hermitian manifold of dimension $2m+1\geq 7$ with $\nabla\bt \neq 0$. Then 
 	\begin{itemize}
 		\item[1.] $\hol(\nabla)=\su(m)$ if and only if $(M,J,\theta)$ is pseudo-Einstein.
 		\item[2.] $\hol(\nabla^{TM})=\su(m)$ if and only if $\rho=0$.
 	\end{itemize}
 \end{prop}

\begin{proof} The proof follows from Theorem \ref{ThClass}, \cite[Theorem~4.1]{Leitner18}, and  
 \cite[Corollary~4.2]{Leitner18}. \end{proof}



\subsection{Example}

We give here a construction for strictly pseudoconvex CR manifolds $M^{2m+1}$ with pseudo-Einstein structure $\theta$  
and  non-trivial torsion $\bt\neq 0$. These pseudo-Hermitian spaces $(M,\theta)$ have horizontal holonomy $\su(m)$.
The construction is not explict, but happens locally by a small 
arbitrary deformation of any strictly 
pseudoconvex hypersurface in complex vector space $\mathbb{C}^{m+1}$. 
Recall that vanishing torsion for a pseudo-Hermitian structure $\theta$
means that the corresponding Reeb vector is a transverse symmetry  of the CR-structure 
and vice versa. Hence, any pseudo-Hermitian structure on a CR-manifold $M$ without transverse symmetries has non-vanishing torsion.
Thus, we aim to carry out local deformations, which eliminate any 
given transverse symmetry. To the same time, results by 
Lee (cf. \cite{Lee}) guarantee the local existence of pseudo-Einstein structures for any strictly pseudoconvex hypersurface in  $\mathbb{C}^{m+1}$, 
in particular, for any such deformation within $\mathbb{C}^{m+1}$.

To make the deformation more precise, let $M'$ be any strictly pseudoconvex hypersurface in $\mathbb{C}^{m+1}$.
We choose some point $p\in M'$ and disjoint small open balls $U$ and $V$ in $\mathbb{C}^{m+1}$ with $p$ on their 
boundaries, which cover $M'$ in a neighborhood of $p$. We deform the restriction of the hypersurface $M'$ to the 
open ball  $V$ in normal direction 
of $M'$ (such that the deformation is smooth up to the boundary 
of  $M\cap V$ and trivial on the boundary and beyond).
If this deformation is small 
the resulting hypersurface $M$ in $\mathbb{C}^{m+1}$ is 
still strictly  pseudoconvex; and, if the deformation 
is non-trivial (at every point 
of $M'\cap V$), then $M\cap M'\cap V$ is empty, whereas $M$ and $M'$ coincide in the closure of  $U$,
in particular, at $p\in M$.  Note that any transverse symmetry $X$ of $M'$ is  
the restriction of some holomorphic vector field $\mathcal{X}$ of $\mathbb{C}^{m+1}$, at least locally in some small neighborhood $W'$
of $p$ in $\mathbb{C}^{m+1}$. 
(In fact, in \cite{Jac} it is argued that the existence of any transverse
symmetry on a CR manifold gives rise to a local embedding into  $\mathbb{C}^{m+1}$ as real hypersurface. This embedding 
construction is explicit and shows
that the transverse symmetry is the restriction of some 
holomorphic vector field.)
Obviously, since $M=M'$ in $U$, the restriction of the holomorphic vector field $\mathcal{X}$ cannot be tangential to the non-trivial deformation  $M$
of $M'$ in $V$. Thus, there exist no transverse symmetries on the deformation $M$ in the neighborhood $W'$ around $p$.
As already mentioned, the results in \cite{Lee} 
guarantee  a pseudo-Einstein structure $\theta$ 
on the deformed hypersurface $M$ in a small neighborhood $W\subseteq W'$ around 
$p$. Thus, 
the deformation $M\cap W$ of $M'\cap W$ around $p$ in $\mathbb{C}^{m+1}$
with $\theta$ is pseudo-Einstein and has non-trivial torsion.
Since there are no transverse symmetries, the construction 
is not locally sub-symmetric, and thus the  horizontal holonomy algebra is $\su(m)$.   


\section{Contact sub-Riemannian spaces with $\nabla\d\theta =0$}\label{secparaldtheta}

In this section we consider a generalization of pseudo-Hermitian  spaces: contact sub-Riemannian spaces with $\nabla\d\theta =0$. 



\begin{theorem}\label{dtheta0theo}
 Let $(M,\theta,g)$ be a contact sub-Riemannian space with $\dim M=2m+1\geq 7$. If  $\bt\neq 0$ and $$\nabla\d\theta =0,$$ then, after multiplying $g$ by a constant, 
	 	$(M,\theta,g)$ is pseudo-Hermitian.
	\end{theorem}

\begin{proof}Define endomorphism field $\psi:\D\to \D$  by the equality
\begin{equation}\label{defpsi} g(X,Y)=\d\theta(X,\psi (Y)),\end{equation}
for $X,Y\in \D$. Then $\psi$ is skew with respect to $g$ and consists of isomorphisms. Since $\nabla g=0$ and $\nabla \d\theta=0$ we also have that $\nabla \psi=0$.

For the Wagner connection $\nabla^W$, recall from Theorem~\ref{thholW} that $\hol_x(\nabla)=\hol_x(\nabla^W)$. Hence, the condition $\nabla\d\theta =0$ implies that $\nabla_\xi^W (\d\theta)=0$, which implies
\[ 0=\left[ N^W,\psi\right](X) +[\xi, \psi (X)]-\psi([\xi,X]).\]
With $N^W=\bt+C$ and using that
\begin{eqnarray*}
0&=&\d^2\theta(\xi,\psi(X),\psi(Y)\ =\  -\L_\xi g(X,\psi(Y) )+g(\psi([\xi,X])-[\xi,\psi(X)],Y)
\\
&=&
g\big(2 \,\psi \circ \bt (X),Y\big)
-g\big( [\xi,\psi(X)] - \psi([\xi,X]),Y\big),
\end{eqnarray*}
so that 
\begin{equation}\label{psibtC}
\bt\circ\psi+\psi\circ\bt {-}\psi\circ C{ +}C\circ\psi=0.\end{equation}
The proof now continues with two lemmas.

\begin{lem}\label{psiinv-lem}
 The holonomy algebra $\hol_x(\nabla)$  does not preserve any proper $\psi_x$-invariant subspace of $D_x$.
\end{lem}

\begin{proof} 

 Suppose that $\hol_x(\nabla^W)$ preserves a proper $\psi_x$-invariant subspace $D^1_x\subseteq D_x$. 
Then it preserves also the   $g$-orthogonal complement $\D^2_x$ to  $\D^1_x$ in $\D_x$, which is also   $\psi_x$-invariant.  Note that if $X_1\in \D^1_x$ and $X_2\in \D^2_x$, then 
\begin{equation}\label{dtheta12}
\d\theta(X_1,X_2)=-g(X_1,\psi_x^{-1}( X_2))=0.\end{equation}
 This implies that the restrictions of $\d\theta$ to $\D^1_x$ and $\D^2_x$ are non-degenerate.
 Therefore, with $m\ge 3$, we may assume that $\dim \D^1_x\geq 4$.

 Let $X_1,Y_1\in \D^1_x$, and $X_2,Y_2\in \D^2_x$.  Using the Bianchi identity 
 $$\underset{X Y Z}{\Su} R^{W}(X,Y)Z=\underset{X Y Z}{\Su} \d\theta(X,Y)(\bt+C) (Z)$$
 applied to $X=X_1$, $Y=Y_2$ and $Z=X_2$, equation~(\ref{dtheta12}),
 and the fact that $R^W$ preserves the $\hol_x(\nabla^W)$-invariant subspaces,  we get
$$g(R^W(X_1,Y_1)X_2,Y_2)=d\theta(X_1,Y_1)g(\bt X_2+C X_2,Y_2).$$
The symmetrization over $X_2$ and $Y_2$ implies
$$\d\theta(X_1,Y_1)g(\bt X_2,Y_2)=0.$$
Since the restriction of $\d\theta$ to $\D^1_x$ is non-degenerate, we conclude that
$$g(\bt \D^2_x,\D^2_x)=0,$$
and similarly,
$$g(\bt \D^1_x,\D^1_x)=0.$$
Let $X_1,Y_1,Z_1\in \D^1_x$, and $X_2\in \D^2_x$. Then the above Bianchi identity yields
$$ 
\underset{X_1 Y_1 Z_1}{\Su}\d\theta(X_1,Y_1)g(\bt Z_1+C Z_1,X_2)=0.
$$
Since $\dim \D^1_x\geq 4$, for each $Z_1$, there exist $X_1,Y_1$ such that $$\d\theta(X_1,Y_1)\neq 0,\quad \d\theta(X_1,Z_1)=d\theta(Y_1,Z_1)=0.$$
We conclude that $$g(\bt Z_1+C Z_1,X_2)=0,\quad \text{for all } Z_1\in \D^1_x,\,X_2\in \D^2_x.$$
Since $\psi_x$ is an isomorphism, we can use  \eqref{psibtC} to get
$$g(\psi \bt Z_1-\psi C Z_1,X_2)=0,\quad \text{for all } Z_1\in \D^1_x,\,X_2\in \D^2_x,$$
so that
$$g( \bt Z_1- C Z_1,X_2)=0,\quad \text{for all } Z_1\in \D^1_x,\,X_2\in \D^2_x.$$
We conclude that $\bt=0$, which gives a contradiction. This proves the lemma. \end{proof}

\begin{lem}\label{Jlem}
The skew isomorphism  $\psi$ defined by \eqref{defpsi} is, after the multiplication of $g$ by a constant, a $\nabla$-parallel complex structure.
\end{lem}
\begin{proof}
Since $\nabla\d\theta=0$, we have that $\nabla\psi=0$ and hence $\hol_x(\nabla) $ commutes with $\psi$.

If $\hol_x(\nabla)$ is  irreducible, then by Schur's lemma the symmetric isomorphism $\psi_x$ satisfies that $\psi^2_x=\lambda\id_{\D_x}$ for $\lambda\in \Real$. Since $\psi$ is skew-symmetric, $\lambda=-\mu^2$, $\mu>0$, so that $J:=\tfrac{1}{\mu} \psi$ is a $\nabla$-parallel complex structure. 

If $\hol_x(\nabla)$ is not irreducible, then  $\hol_x(\nabla)$ preserves  a non-trivial proper subspace $\D^1_x$  of $ \D_x$ and  also the orthogonal complement $(\D^1_x)^\bot\subseteq \D_x$. We may assume that $\dim \D^1_x\geq m$. Since $\psi$  is $\nabla$-parallel,  
$\hol_x(\nabla)$ preserves also the intersection $\psi \D^1_x\cap \D^1_x$, which is $\psi$-invariant. Lemma~\ref{psiinv-lem} then applies and we get that $\psi \D^1_x\cap \D^1_x=\{0\}$, and $\dim \D^1_x= m$. Consequently,  $$\D_x=\D^1_x\oplus \psi \D^1_x$$  is a $g$-orthogonal decomposition preserved by $\hol_x(\nabla)$. 
Since this holds for any invariant subspace, the representation of $\hol_x(\nabla)$ on $\D^1_x$ is irreducible.
Moreover, the $g$-symmetric isomorphism $$\psi^2|_{\D^1_x}:\D^1_x\to \D^1_x$$ commutes with $\hol_x(\nabla)$, so that by Schur's lemma $$\psi^2|_{D^1_x}=\lambda\id_{\D^1_x},$$
for $\lambda\in \Real$.
Again, since $\psi$ is skew-symmetric, $\lambda=-\mu^2$, $\mu>0$.
We conclude that $$J_x=\frac{1}{\mu}\psi_x$$ is a  complex structure commuting with  $\hol_x(\nabla)$
\end{proof}

For the proof of the theorem, it remains to show that the obtained complex structure $J$ is not only parallel for the Schouten connection $\nabla$ and the Wagner connection $\nabla^W$, but also for the adapted connection $\nabla^\bt$. 
%
%
%
For this we note that  $$\bt\circ J+J\circ\bt\in\u(\D).$$
Indeed, since $J\in\so(\D)$ and $\bt$ is symmetric, it holds  $\bt\circ J+J\circ\bt\in\so(\D)$, and moreover$$J\circ(\bt\circ J+J\circ\bt)=
J\circ\bt\circ J-\bt=(\bt\circ J+J\circ\bt)\circ J.$$ 
For the endomorphism $N^W=\bt+C$ defining the Wagner connection, with $C\in \so(\D)$, we split  $C=C_1+C_2$, where $C_1\in\u(\D)$ and $C_2$ belongs to the orthogonal complement  $\u(\D)^\bot$ to $\u(\D)$ in $\so(\D)$. 
It is clear that $J$ preserves $\u(\D)^\bot$. With this  and with $\bt\circ J+J\circ\bt\in\u(\D)$, equation~\eqref{psibtC} implies that $[J,C_2]=0$, so that that $C_2=0$. Now with $C\in\u(\D)$, we get $$\nabla^\bt_\xi J=\nabla^W_\xi J-[C,J]=0.$$ Thus, $(M,\theta,g)$ is pseudo-Hermitian.
\end{proof}


\section{Classification of sub-symmetric spaces in terms of holonomy}
\label{subsymholclass-sec}

In this section we will deduce the classification of simply connected sub-symmetric spaces from the classification of the holonomy algebras. In particular, we will show the following result.
\begin{theorem}\label{subsymholclass-thm}If $(M,g,\theta)$ is a simply-connected sub-Riemannian symmetric space of dimension~$\ge 7$, then $M$ is given by one of the entries in Table~1 with the corresponding holonomy algebras as listed. 
In the table, HRSS stands for Hermitian Riemannian symmetric space, and $\t$ is as in Theorem~\ref{thholnablaKcontact}. 

\begin{table}[h]
	\caption{Sub-symmetric spaces $M$ and their holonomy for $\dim M\geq 7$}
	\begin{tabular}{|c|c|c|c|}\hline
		$\bt$ & Sub-symmetric space & $\hol(\nabla)$ & $\hol(\nabla^\bt)$\\ \hline
		$\bt=0$ & $H^{2m+1}=$  Heisenberg group of dimension $2m+1$ & $\{0\}$ & $\{0\}$\\\cline{2-4}
		&  $\mathbb{S}^1$-bundle over HRSS $N$ without flat factor& $\hol(N)/\t$ & $\hol(N)$\\ \cline{2-4}
		& twisted product of $H^{2m+1}$ and $\mathbb{S}^1$-fibration over HRSS $N$ & $\hol(N)$ & $\hol(N)$\\ \hline
		$\bt\neq 0$ & $\SO(m+2)/\SO(m)$, $\SO(2,m)/\SO(m)$  & $\so(m)$ & $\so(m)\oplus\u(1)$\\
		&  $\SO(1,m+1)/\SO(m)$ with $scal^\bt\neq 0$ &  & \\
		& $\SO(m+1)\ltimes\Real^{m+1} /\SO(m)$, $\SO(1,m)\ltimes\Real^{1,m} /\SO(m)$  &  & \\ \cline{2-4}
		&  $\SO(1,m+1)/\SO(m)$ with $scal^\bt= 0$ & $\so(m)$ & $\so(m)$ \\ \hline 
	\end{tabular}
\end{table}
\end{theorem}
We should point out that the table in parts can be derived from \cite{BFG,FG,FGR}. More explicitly,  \cite[Theorem 5.1]{FG} give the classification in the case when $\h$ acts irreducibly on $\p$. Then \cite[Theorem 5.1]{FGR} provides the relation between the horizontal and the adapted holonomy as in~(\ref{symhol}) and based on this  the classification is extended to the case that the adapted holonomy is irreducible \cite[Theorem~6.1]{FGR}. Finally, in  \cite{BFG} also the reducible case is considered and a table similar to ours is provided (assuming that the authors of \cite{BFG}  refer to the holonomy of the adapted connection in their table). We want to stress that
  our proof does not use any of these results, but rather derives the table from the holonomy classification in our Theorem~\ref{ThClass}. Moreover, our table also gives the relation between horizontal and adapted holonomy, which is necessary for a complete classification.  Finally we would like to point out that in the case $\bt\neq 0$, while $\D$ is uniquely determined, there is a $2$-parameter family of metrics on $\D$, see  \cite[Theorem~6.1]{FGR}. This will also be visible from our proof below. 
  
%
%
%
%
%

\medskip

For the proof we proceed in several steps and 
first we recall some facts.

\subsection{Sub-symmetric spaces as $\mathbb{S}^1$-bundles over HRSS} 
Let $N=G/H$ be a HRSS of dimension $2m$. Then $H=K\cdot \mathbb{S}^1$, $K\subseteq\SU(m)$, $\mathbb{S}^1=\Un(1)$. Let $M$  be the bundle of units of the canonical complex line bundle over $N$, which is a $\mathbb{S}^1$-bundle over $N$.
 The Lie group $G$ acts transitively on $M$ with the stability subgroup $K$, i.e., $M= G/K$. Let $\theta$ be the connection 1-form on
 the bundle $M\to N$. Then $\theta$ is a contact form on $M$. The Riemannian metric on $N$ induces a sub-Riemannian metric on  $\D=\ker(\theta)$. 
 The circle fibres  correspond to the flow of the Reeb vector field $\xi$. 
Moreover, the distribution $\D=\ker(\theta)$ is $G$-invariant and the $\mathrm{Ad}_K$-invariant complement to the fibres.

\subsection{$\mathbb{Z}_2$-gradings of pseudo-orthogonal Lie algebras}
		 Let $\mathfrak{b}$ be one of the Lie algebras
	$\so(m+2)$ and $\so(2,m)$.
	Let $\epsilon=+$ and $\epsilon=-$ for the Lie algebras $\so(m+2)$ and $\so(2,m)$, respectively. Let $f_1,f_2,e_1,\dots,e_m$ be an orthonormal basis of the space $\Real^{m+2}$ (if $\epsilon=+$) or of the space $\Real^{2,m}$ (if $\epsilon=-$). Then there is a Lie algebra $\mathbb{Z}_2$-grading 
	$$\mathfrak{b}=(\so(m)\oplus\u(1))\oplus\Real^{2m},$$
	$$\Real^{2m}=\Real^m\otimes \Real^2,\quad \Real^m=\mathrm{span}\{e_1,\dots,e_m\},\quad 
	\Real^2=\mathrm{span}\{f_1,f_2\}$$ and it holds
	$$[X\otimes f_\alpha, Y\otimes f_\beta]_{\h}=\epsilon\delta_{\alpha\beta} X\wedge Y+g(X,Y) f_\alpha\wedge f_\beta,\quad \alpha,\beta=1,2.$$
	Note that with respect to the decomposition $\Real^{2m}=\Real^m\otimes f_1\oplus \Real^m\otimes f_2$ it holds
	$$f_1\wedge f_2=\epsilon \left( \begin{matrix} 0& -\1 \\ \1&0 \end{matrix}\right).$$
	
	Let $\mathfrak{b}$ be one of the Lie algebras
	$\so(m+1)$ and $\so(1,m)$.
	Let $\epsilon=+$ and $\epsilon=-$ for the Lie algebras $\so(m+1)$ and $\so(1,m)$, respectively. Let $e_\epsilon,e_1,\dots,e_n$ be an orthonormal basis of the space $\Real^{m+1}$ (if $\epsilon=+$) or of the space $\Real^{1,m}$ (if $\epsilon=-$). Then there is a Lie algebra $\mathbb{Z}_2$-grading
	$$\mathfrak{b}=\so(m)\oplus\Real^{m},\quad 
	\Real^{m}= \mathrm{span}\{e_1,\dots,e_m\}\otimes e_\epsilon$$ and it holds
	$$[X\otimes e_\epsilon, Y\otimes e_\epsilon]_{\h}=\epsilon X\wedge Y.$$
	Let $p,e_1,...,e_m,q$ be a Witt basis of the Minkowski space $\Real^{1,m+1}$. 
	It defines the $\mathbb{Z}_2$-grading
	$$\so(1,m+1)=(\so(m)\oplus\Real p\wedge q)\oplus(p\wedge \Real^m\, \oplus\, q\wedge \Real^m).$$ It holds that $$[p\wedge X, q\wedge Y]=X\wedge Y-g(X,Y) q\wedge p, \quad [p\wedge X, p\wedge Y]=[q\wedge X, q\wedge Y]=0.$$

	\subsection{Classification (proof of Theorem~\ref{subsymholclass-thm})}

Let $(M,\theta,g)$ be a simply-connected sub-symmetric contact sub-Riemannian space with $\dim M=2m+1\geq 7$.
We will consider two cases.

\subsubsection{Case $\bt=0$} Suppose that $\hol(\nabla^\bt)=\{0\}$.  Using the information from Section \ref{secLocsubsym}, we get 
$$\g=\p\oplus\left<\xi\right>$$ with 
$$[X,Y]=\Theta(X,Y),\quad [\xi,X]=0,\quad \text{for all } X,Y\in\p.$$ We see that $\g$ is the Heisenberg algebra, and $(M,\theta,g)$ is the Heisenberg group
 $H^{2m+1}$. 
 
 Suppose that $\hol(\nabla^\bt)\neq \{0\}$.
Since $(M,\theta,g)$ is sub-symmetric, it holds $\nabla\d\theta =0$. From this and the equality $\bt=0$ it follows that $\nabla^\bt(\d\theta)=0$. We conclude that in decomposition \eqref{deRD} and \eqref{deRh}, each Lie algebra $\h_i$ is contained in $\u(D^i_x)$, and $\hol(\nabla^\bt)\subseteq\u(\D_x)$ with respect to a complex structure $J_x$. It holds $R^\bt_x\in\R(\hol(\nabla^\bt))$. The decomposition \eqref{deRD} defines the decomposition
$$R^\bt_x=R_1+\cdots+R_r,$$
where, for each $i$, $R_i\in\R(\h_i)$.
Since $\hol(\nabla^\bt)$ annihilates $R_x$, each $\h_i$ annihilates $R_i$. We see that the pair $(\h_i,R_i)$ defines a simply-connected Hermitian (K\"ahler) symmetric space $$N_i=F_i/H_i,$$
where $F_i$ is the simply-connected Lie group
with the $\mathbb{Z}_2$-graded Lie algebra 
$$\h_i\oplus \D^i_x,\quad \text{ where $R_i(\D^i_x,\D^i_x)=\mathrm{span}\{ R_i(X,Y)\mid X,Y\in \D^i_x\}$, }$$ 
and $H_i\subseteq F_i$ is the connected Lie subgroup corresponding to $\h_i$.  
The restriction  of the Ricci-form  $\rho$ of $R^\bt$ to $\D^i_x$ coincides with the Ricci-form $\rho^i$, which satisfies 
$$\rho^i=c_i \,\d\theta|_{\D^i_x},\quad \text{ with a constant }c_i\neq 0.$$  
According to Theorem \ref{thCriter}, it holds 
$\hol(\nabla^\bt)=\hol(\nabla)$ if and only if the space $\D_x^0$ is trivial. 

If $\D_x^0=0$, then $\p=\D_x$, $\k=\hol_x(\nabla)$, $\h=\hol_x(\nabla^\bt)=\k\oplus\t_x$,
$$[X,Y]=-R^W_x(X,Y)+\d\theta(X,Y)\xi,\quad \text{for all } X,Y\in \D_x,$$
as in~(\ref{Liebracket}), 
and we see that  $(M,\theta,g)$ is an $\mathbb{S}^1$-bundle over 
the product $$N=N_1\times\cdots\times N_r.$$

If $\D_x^0\neq 0$, then $\k=\hol_x(\nabla)=\hol_x(\nabla^\bt),$
$\h=\k\oplus\Real\xi$, 
$$[X,Y]=d\theta(X,Y)\xi,\quad \text{for all } X,Y\in \D^0_x,$$
$$[X,Y]=-R^W_x(X,Y)+d\theta(X,Y)\xi,\quad \text{for all } X,Y\in \D^i_x,\quad i\geq 1.$$
And we obtain the twisted product of a Heisenberg group with $N$.

\subsubsection{Case $\bt\neq 0$} 
From the results of Section \ref{secparaldtheta} it follows that (after a homothetic change of $g$) $(M,\theta,g)$ is a pseudo-Hermitian space. 
From Proposition \ref{PropSym} it follows that $\hol(\nabla)=\so(m)$.
Fix a point $x\in M$. Let, as in Section \ref{secLocsubsym}, $\p=\D_x$. In what follows we will consider all tensors to be defined at the point $x$.
As in the proof of Theorem \ref{ThClass}, we get
$$R^\bt=R_0+R_\bt,$$ where
$R_0\in\R(\so(m)\oplus\u(1))$, and
$$\bt=\lambda\left( \begin{matrix} \1& 0\\0&-\1\end{matrix}\right),\quad\lambda > 0,$$
with respect to an orthonormal basis $e_1,\dots,e_m,Je_1,\dots, Je_m$ of $\p$.
The space $\R(\so(m)\oplus\u(1))$ is one-dimensional \cite{berger55}  and it is spanned by the curvature tensor $R_1$ of the Riemannian symmetric space $\SO(m+2)/(\SO(m)\cdot\SO(2))$. Consider the above $\mathbb{Z}_2$-grading of $\so(m+2)$. We identify $\p$ with $\Real^{2m}$ in such a way that
$$J(X\otimes f_1)= X\otimes f_2,\quad J(X\otimes f_2)= -X\otimes f_1,\quad \text{ for all }X\in \Real^m.$$
In particular, $$J=f_1\wedge f_2\in\so(m).$$
Under this identification,
$$R_1(X\otimes f_\alpha, Y\otimes f_\beta)=-[X\otimes f_\alpha, Y\otimes f_\beta]_{\so(m+2)}=-\delta_{\alpha\beta} X\wedge Y-g(X,Y) f_\alpha\wedge f_\beta,\quad \alpha,\beta=1,2.$$ Let $\mu\in\Real$ be such that 
$$R_0=\mu R_1.$$
As in Section \ref{secLocsubsym}, we have the 
$\mathbb{Z}_2$-grading
$$\g=(\so(m)\oplus\left<\xi\right>)\oplus \p,\quad \p=\Real^{2m}.$$
Since $\bt\neq 0$, we have $\ad_\xi =N^W=\bt+C\not=0$,  so we identify $\xi$ with $\bt+C$.
Thus it holds 
\begin{equation}\label{bracketgR}
\begin{array}{rcl}
[X\otimes f_\alpha, Y\otimes f_\beta]_\g&=&-R^W(X\otimes f_\alpha, Y\otimes f_\beta)+d\theta(X\otimes f_\alpha, Y\otimes f_\beta)(\bt +C)\\[1mm]
&=&-R(X\otimes f_\alpha, Y\otimes f_\beta),\end{array}\end{equation}
as well as
\begin{equation}\label{RXfaYfa}R(X\otimes f_\alpha, Y\otimes f_\alpha)=R^\bt(X\otimes f_\alpha, Y\otimes f_\alpha)=-\mu X\wedge Y,\quad \alpha=1,2,\end{equation}
and 
\begin{eqnarray*}
R(X\otimes f_1, Y\otimes f_2)&=&R^\bt(X\otimes f_1, Y\otimes f_2)-d\theta(X\otimes f_1, Y\otimes f_2)\bt\\
&=&-\mu g(X,Y)f_1\wedge f_2+\lambda X\wedge Y+-g(X,Y)\bt.\end{eqnarray*}
Next, $$C=-\frac{1}{2m}R^\bt(J)=-\frac{1}{2m}R_0(J) =\mu f_1\wedge f_2=\mu J.$$
This implies
\begin{equation}\label{RXf1Yf2}R(X\otimes f_1, Y\otimes f_2)=\lambda X\wedge Y-g(X,Y)(\bt+C).\end{equation}
Note that
$$scal^\bt=\mu\cdot scal(R_1)=2\mu m^2.$$ 
We see that
$$\ad_\xi=\bt+C=\left(\begin{matrix} \lambda\1& -\mu \1\\\mu \1& -\lambda \1 \end{matrix}\right).$$
Recall that $\lambda>0$. Similarly as in \cite{FGR} we consider several cases.

\bigskip

{\bf Case $\mu=0$.} 
The equalities \eqref{bracketgR}, \eqref{RXfaYfa}, and \eqref{RXf1Yf2}   show that the map $$X\otimes f_1\mapsto \sqrt{\lambda}\ p\wedge X,\quad  X\otimes f_2\mapsto \sqrt{\lambda} \ q\wedge X$$
defines the isomorphism $\g\cong\so(1,m+1)$. Thus,
$M=G/K$ is the space $\SO(1,m+1)/\SO(m)$. 

\bigskip

{\bf Case $\mu=\lambda$.} It holds
$$\ker (\bt+C)=\{X\otimes f_1+X\otimes f_2\mid X\in\Real^m\}.$$ Let  $\Real^{m+1}:=\ker (\bt+C)\oplus\left<\xi\right>\subseteq\g$. From \eqref{RXfaYfa} and \eqref{RXf1Yf2} it follows that $\Real^{m+1}\subseteq \g$ is a commutative ideal, and 
$$[X\otimes f_1-X\otimes f_2,Y\otimes f_1-Y\otimes f_2]_\g=4\mu X\wedge Y.$$
This implies that 
$$\g\cong \g_1\ltimes \Real^{m+1},\quad \g_1\cong\so(m+1),$$ where
$$\g_1=\so(m)\oplus L_1,\quad L_1=\{X\otimes f_1-X\otimes f_2|X\in\Real^m\}.$$
This shows that 
$M=G/K$ is the space $\SO(m+1)\ltimes\Real^{m+1}/\SO(m)$. 

\bigskip

{\bf Case $\mu=-\lambda$.} This case is similar to the previous one. It holds that
$$\g\cong \so(1,m)\ltimes\Real^{1,m},$$
where 
$$\Real^{1,m}=\ker (\bt+C)\oplus\left<\xi\right>,\quad 
\ker (\bt+C)=\{X\otimes f_1-X\otimes f_2|X\in\Real^m\}.$$
We see that 
$M=G/K$ is the space $\SO(1,m)\ltimes\Real^{1,m}/\SO(m)$. 

\bigskip

{\bf Case $0\neq |\mu|<|\lambda|$.} In that case, $\bt+C$ has real eigenvalues $\pm\sqrt{\lambda^2-\mu^2}$, and 
we obtain the decomposition 
$$\Real^{2m}=L_1\oplus L_2$$ into the direct sum eigenspaces, where
$$L_{1,2}=\left.\left\{ \varphi_{1,2}(X):=X\otimes f_1+\frac{\lambda\mp \sqrt{\lambda^2-\mu^2}}{\mu}X\otimes f_2\right| X\in\Real^m   \right\}.$$
It holds $$[L_\alpha,L_\alpha]=0,\quad \alpha=1,2,$$
$$R(\varphi_1(X),\varphi_2(X))=\frac{2}{\mu}(\lambda^2-\mu^2)\left(X\wedge Y-\frac{1}{\sqrt{\lambda^2-\mu^2}}g(X,Y)(\bt+C)\right).$$
With respect to the basis $\varphi_1(e_1),\dots,\varphi_1(e_n),\varphi_2(e_1),\dots,\varphi_2(e_n)$ of  $\Real^{2m}$ it holds
$$\frac{1}{\sqrt{\lambda^2-\mu^2}}(\bt+C)=\left(\begin{matrix}\1&0\\0&-\1\end{matrix}\right),$$
and we see that $$\g=(\so(m)\oplus\left<\xi\right>)\oplus (L_1\oplus L_2)$$ is isomorphic to $\so(1,m+1)$.
We see that 
$M=G/K$ is the space $\SO(1,m+1)/\SO(m)$. 

\bigskip

{\bf Case $|\mu|>\lambda$.}
Note that $$\frac{1}{\mu^2-\lambda^2}(\bt+C)^2=-\id_{\Real^{2m}}.$$
Let $$L_{1}=\left.\left\{ \varphi_{1}(X):=X\otimes f_1\right| X\in\Real^m   \right\},$$
$$L_{2}=\left.\left\{ \varphi_{2}(X):=\frac{1}{\sqrt{\mu^2-\lambda^2}}(\bt+C)\varphi_1(X)=\frac{1}{\sqrt{\mu^2-\lambda^2}}(\lambda X\otimes f_1+\mu X\otimes f_2)\right| X\in\Real^m   \right\}.$$
It holds 
$$R(\varphi_\alpha(X),\varphi_\alpha(X))=-\mu\frac{1}{\sqrt{\mu^2-\lambda^2}} X\wedge Y,\quad \alpha =1,2,\,\,X\in\Real^m,$$
$$R(\varphi_1(X),\varphi_2(X))=-\mu\frac{1}{\sqrt{\mu^2-\lambda^2}}g(X,Y)(\bt+C).$$
This shows that if $\mu>0$, then $\g$ is isomorphic to $\so(m+2)$, and if 
$\mu<0$, then $\g$ is isomorphic to $\so(2,m)$

\section{Parallel horizontal spinors}\label{secSpinors}

Let $(M,J,\theta)$ be a simply connected pseudo-Hermitian space of dimension $n=2m+1$ with contact distribution $\mathcal{D}$. 
The form $d\theta^m$ defines an orientation on $\D$. 
Let $P_{\SO}$ be the $\SO(2m)$-principle bundle of positively oriented frames on $D$.  
A {\it spin-structure} on $\D$ is a  reduction of $P_{\SO}$ to a $\Spin(2m)$-principle bundle $P_{\Spin}$ 
corresponding to the two-fold covering $$\lambda:\Spin(2m)\to\SO(2m).$$
Suppose that a spin-structure on $\D$ is fixed.
Denote by $\Delta_{2m}$ the complex 
$\Spin(2m)$-spinor module. {\it The spinor bundle} $S$  is defined as the associated bundle,
$$S=P_{\Spin}\times_\lambda \Delta_{2m}.$$
The horizontal connection $\nabla$ is lifted to a horizontal  spinor derivative
$$\nabla^S:\Gamma(\D)\times\Gamma(S)\to\Gamma(S).$$
The holonomy algebra of this connection is the representations of $\hol_x(\nabla)$ 
on $S_x\cong\Delta_{2m}$. By the holonomy principle, there exists a non-zero spinor 
$s\in \Gamma(S)$ such that $\nabla^S s=0$ if and only if $\hol_x(\nabla)$  preserves a non-zero element $s_x\in S_x\cong\Delta_{2m}$.

Parallel horizontal spinors and their relation to holonomy on pseudo-Hermitian manifolds and on K-contact sub-Riemannian manifolds 
were studied in \cite{Leitner18} and \cite{GalHolK}. We use these results and the results of the current paper, 
and obtain the following cases.

\begin{theorem}\label{thparalspin} Let $(M,J,\theta)$ be a simply connected pseudo-Hermitian space of dimension $n=2m+1\geq 7$ 
	with a fixed spin structure on the contact distribution $\D$. 
	Then $(M,J,\theta)$ admits a non-zero parallel horizontal spinor if and only if one of the following conditions holds:
	\begin{itemize}
		\item[{\bf 1.}] $\bt\neq 0$ and $\hol(\nabla)$ is either $\su(m)$ or $\so(m)$; 
		\item[{\bf 2.}]  $\bt=0$, $\hol_x(\nabla)\neq\hol_x(\nabla^\bt)$, and there 
		exists a $\nabla$-parallel $g$-skew-symmetric complex structure $I$ on $\D$ such that $\hol_x(\nabla)\subseteq\su(\D_x,g_x,I_x)$; 
		\item[{\bf 3.}] $\bt=0$, $\hol_x(\nabla)=\hol_x(\nabla^\bt)$, and each 
		Lie algebra in decomposition \eqref{deRh} is one of $\su(k)$, or $\mathfrak{sp}(k)$.
	\end{itemize}
	In cases 1 and 2 the dimension of the space of parallel spinors is $2$.
\end{theorem}

\begin{proof} 
It is well-known that the $\u(m)$-module $\Delta_{2m}$ admits the invariant decomposition 

\begin{equation}\label{Delta2n}\Delta_{2m}=\bigoplus_{k=0}^m\Lambda^k\Co^m\end{equation}
such that $\Co^m$ is the standard $\su(m)$-module, and the complex structure $J\in\u(m)$ acts on $\Lambda^k\Co^m$
as multiplication by $(m-2k)i$.

Suppose that $\bt\neq 0$. Consider the holonomy algebras $\hol_x(\nabla)$  from Theorem \ref{ThClass}. 
By the result of Wang \cite{Wang}, $\u(m)$ and $\so(m)\oplus\so(2)$ do not annihilate any non-zero element from $\Delta_{2m}$. 
From \cite[Lemma 5]{GalHolK} it follows that   the subspace of $\Delta_{2m}$   annihilated by $\so(m)$ is 
the sum $\Lambda^0\Co^m\oplus \Lambda^m\Co^m$ of the extremal bundles of the spinor representation. Hence, the dimension of parallel
spinors is 2. The same statement is true for $\su(m)$. Only the extremal spinors are annihilated. 

Suppose that $\bt=0$ and  $\hol_x(\nabla)\neq\hol_x(\nabla^\bt)$. The corresponding statements of the theorem follow  
from \cite[Theorem 8]{GalHolK}.

Suppose that $\bt=0$ and  $\hol_x(\nabla)=\hol_x(\nabla^\bt)$. In that case,  
$\hol_x(\nabla)$ is the holonomy algebra of  a Riemannian manifold, and the result follows from \cite{Wang}.
\end{proof}


Note that any pseudo-Einstein space $(M,J,\theta)$ with or without torsion 
$\bt$ of dimension $n=2m+1$ satisfies $\hol_x(\nabla)\subseteq \su(2m)$.
Hence, any  simply connected pseudo-Einstein space $(M,J,\theta)$ admits 
non-zero parallel horizontal spinors in the extremal bundles
$\Lambda^0\Co^m\oplus \Lambda^m\Co^m$. It is well known that any strictly 
pseudoconvex CR manifold $(M^{2m+1},\mathcal{D},J)$,
which is embeddable into complex space $\Co^{m+1}$, admits a pseudo-Einstein 
structure $\theta$. E.g.~any smooth boundary $M^{2m+1}$ of a domain of holomorphy 
in  $\Co^{m+1}$ admits a pseudo-Einstein structure $\theta$ and non-zero parallel
horizontal spinors, no matter whether the torsion $\bt$ vanishes or not (cf. \cite{Leitner18}).

In the second case of Theorem~\ref{thparalspin} the pseudo-Hermitian 
spaces with parallel spinors  are not pseudo-Einstein (although the 
holonomy algebra $\hol_x(\nabla)$ is contained in the special unitary algebra
$\su(\D_x,g_x,I_x)$ for some complex structure $I\neq J$).
An explicit  construction of such a pseudo-Hermitian space $(M,J,\theta)$                  
is given for  $m$ even in Section~7.2.~of \cite{Habib}. This construction
multiplies two Einstein-Sasakian spaces, both of  dimension $m+1$,
of Riemannian and Lorentzian signature, resp., to produce the  Fefferman 
space of some underlying CR manifold $(M^{2m+1},\mathcal{D},J)$ with 
pseudo-Hermitian structure $\theta$. The CR structure $(\mathcal{D},J)$  
decomposes by construction into $(\mathcal{D}_1,J_1)\oplus(\mathcal{D}_2,J_2)$, and
the holonomy algebra $\hol_x(\nabla)$ is isomorphic to 
$\su(\frac{m}{2})\oplus\su(\frac{m}{2})\oplus \Real J$, which is only contained
in $\su(D_x,g_x,I_x)$ for the mutation $I=J_1-J_2$. Some twistor spinors on the 
Fefferman space (which are related to the Killing spinors
of the Einstein-Sasakian spaces) descend to horizontal parallel spinors on $(M,J,\theta)$, which sit
in the middle slot $\Lambda^{m/2}\Co^m$ of   $\Delta_{2m}$ with respect to $J$. 
With respect to the mutation $I$ these spinors are in the exremals 
$\Lambda^0\Co^m\oplus \Lambda^m\Co^m$. The spinors are not $\nabla^\bt$-constant 
in direction of the Reeb vector $\xi$, and the holonomy algebra of the Tanaka-Webster 
connection properly contains  $\hol_x(\nabla)$.


\begin{thebibliography}{10}
	
		\bibitem{A10}  Agricola, I.	Non-integrable geometries, torsion, and holonomy. In Handbook of pseudo-Riemannian Geometry and Supersymmetry, 	IRMA, EMS, 2010, 277--346.

	

\bibitem{AmbroseSinger58}
Ambrose, W. and  Singer, I.~M.
\newblock On homogeneous {R}iemannian manifolds.
\newblock {\em Duke Math. J.}, 25:647--669, 1958.

\bibitem{berger55}
M.~Berger.
\newblock Sur les groupes d'holonomie homog\`ene des vari\'et\'es \`a connexion
affine et des vari\'et\'es riemanniennes.
\newblock {\em Bull. Soc. Math. France}, 83:279--330, 1955.

	\bibitem{Besse} Besse, A.L. Einstein manifolds. Springer-Verlag, 	Berlin-Heidelberg-New York, 1987.


\bibitem{BFG}
Bieliavsky, P., Falbel, E., Gorodski, C. The classification of simply-connected contact sub-Riemannian symmetric spaces. Pacific J. Math. 188 (1999), no.~1, 65--87.



\bibitem{Bryant92}  Bryant, R. L. Classical, exceptional and exotic holonomies: a status report, Actes
de la table ronde de g\'eom\'etrie diff\'erentielle (Luminy, 1992), S\'emin. Congr., 1, Soc.
Math. France, Paris, 1996, 93–165.

\bibitem{bryant99}
 Bryant, R.~L.
\newblock Recent advances in the theory of holonomy.
\newblock {\em S\'eminair BOURBAKI}, 51(861):1--24, 1999.



	
\bibitem{CGJK} Chitour, Y., Grong, E., Jean, F., Kokkonen, P. Horizontal holonomy and foliated manifolds. Ann. Inst.
Fourier (Grenoble) 69(3) (2019), 1047–1086.

\bibitem{C-S} Cleyton, R., Swann, A.  Einstein metrics via intrinsic or parallel torsion. Math. Z. 247  (2004), 513--528.

\bibitem{DT2006} 
Dragomir, S.,  Tomassini, G. Differential Geometry and analysis on CR manifolds. Progress in Math., vol. 246, Birkh\"auser Boston, Inc., Boston,
MA, 2006.


\bibitem{FG} Falbel, E., Gorodski, C. On contact  sub-Riemannian symmetric spaces. 
Ann. Sc. \'Ec. Norm. Sup. 28 (1995), 571-589.	
	
\bibitem{FGR} Falbel, E., Gorodski, C., Rumin, M. Holonomy of sub-Riemannian manifolds.
Int. J. Math. 8 (1997), no. 3, p. 317-344.

\bibitem{Habib} Habib, G., Leitner, F. \newblock  Eigenvalue estimates of the Kohn-Dirac operator
on CR manifolds. \newblock arXiv:2503.22450     

\bibitem{HMCK}
Hafassa, B.,  Mortada, A.,   Chitour, Y.,  Kokkonen, P. Horizontal holonomy
for affine manifolds, J. Dyn. Control Syst. 22 (2016), no. 3,  413--440.


\bibitem{Papa2016} Galaev, S.V. 	Geometric interpretation of the Wagner curvature tensor in the case of a manifold with contact metric structure.
Siberian Math. J. 
57 (2016), no. 3, 498--504.


\bibitem{GalHolK} Galaev, A.S. Holonomy of K-contact sub-Riemannian manifolds. Annali di Matematica Pura ed Applicata, doi: 10.1007/s10231-025-01612-w 




\bibitem{Jac} Jacobowitz, H.  An Introduction to CR Structures. Mathematical Surveys and Monographs 32 A.M.S., 1990.

\bibitem{Joyce07}  Joyce, D. Riemannian holonomy groups and calibrated geometry, Oxf. Grad. Texts
Math., 12, Oxford Univ. Press, Oxford, 2007.

\bibitem{Lee} Lee, J.M., Pseudo-Einstein Structures on CR Manifolds. 
American Journal of Mathematics, Vol. 110, No. 1 (Feb., 1988), pp. 157-178



\bibitem{Leistner} Leistner, T.  On the classification of Lorentzian holonomy
	groups, J.  Differential Geom.  76 (2007), no. 3, 423--484.





\bibitem{Leitner18} Leitner, F. Parallel spinors and basic holonomy in pseudo-Hermitian geometry.  Ann. Glob. Anal. Geom. 55 (2019), 181–196.

\bibitem{Salamon89}
 Salamon, S.~M.
\newblock {\em Riemannian Geometry and Holonomy Groups}, volume 201 of {\em
  Pitmann Research Lecture Notes}.
\newblock 1989.




\bibitem{Tanaka}  Tanaka, N.
\emph{A differential geometric study on strongly pseudoconvex manifolds}, Lectures in Mathematics,
Vol. \textbf{9}, Kinokuniya Book-Store Co., Ltd., Tokyo, 1975

\bibitem{Tanno}  Tanno, S. Variational problems on contact riemannian manifolds, Trans. Amer. Math. Soc. 314 (1989) 349–379.


\bibitem{tricerri-vanhecke83}
Tricerri F., Vanhecke, L. \emph{Homogeneous structures on {R}iemannian
  manifolds}, London Mathematical Society Lecture Note Series, vol.~83,
  Cambridge University Press, Cambridge, 1983. \MR{712664 (85b:53052)}


\bibitem{Wagner41} Wagner, V.V. Geometry of $(n-1)$-dimensional nonholonomic manifold in an $n$-dimensional space. Proc. Sem. on Vect. and Tens. Anal. (Moscow Univ.)  5 (1941), 173--255.



\bibitem{Wang} Wang, M. {Parallel spinors and parallel forms}. Ann.
Global Anal. Geom. 7 (1989), no. 1, 59--68.

\bibitem{Webster}  Webster, S.M. 
\emph{Pseudo-Hermitian structures on a real hypersurface}, 
J. Diff. Geom. \textbf{13} (1978), 25–41.


\end{thebibliography}
\end{document}